\documentclass[12pt, oneside,reqno]{amsart}

\usepackage{amssymb}
\usepackage{amsmath}
\usepackage[normalem]{ulem}
\usepackage{multirow}
\usepackage{array}
\usepackage{graphicx}
\graphicspath{ {./images/} }
\usepackage{xcolor}%I added this package so that I can keep track of my edits in color.
\usepackage{graphicx} 
\usepackage{mathrsfs}
\usepackage[textwidth=14mm]{todonotes}
\usepackage{cancel}
\usepackage{amsthm}
\usepackage{full page}
\usepackage{hyperref}
\allowdisplaybreaks
\usepackage{geometry,mathtools}
\usepackage{enumitem}
\usepackage{comment}
\usepackage[new]{old-arrows}

\usepackage{stackengine}

\newtheorem{thm}{Theorem}[section]
\newtheorem{prop}[thm]{Proposition}
\newtheorem{cor}[thm]{Corollary}
\newtheorem{lem}[thm]{Lemma}
\theoremstyle{definition}

\newtheorem{rem}[thm]{Remark}
\newtheorem{defn}[thm]{Definition}

\reversemarginpar

\newcommand{\C}{\mathbb{C}}

\newcommand{\Z}{\mathbb{Z}}

\DeclarePairedDelimiter{\pp}{(\!(}{)\!)}
\DeclarePairedDelimiter{\bb}{[[}{]]}

\newcommand{\wt}{{\rm wt}}
\newcommand{\res}{\mbox{\rm Res}}

\newcommand{\vac}{\mathbf{1}}

\newcommand{\End}{{\rm End
}}

\newcommand{\hZhu}[1]{\mathrm{Zhu}_{#1}}    %DLM Zhu's algebra
\newcommand{\Zhu}{\mathrm{Zhu}}
\newcommand{\HeZhu}[1]{\mathrm{A}_{#1}}     %He's definition of Zhu's algebra
\newcommand{\MTA}{\mathfrak{A}}
\newcommand{\mcU}{\mathscr{U}}              %envelopping algebra DGK
\newcommand{\NL}{\mathrm{N}_{\mathrm{L}}}
\newcommand{\NR}{\mathrm{N}_{\mathrm{R}}}
\newcommand{\WeylVA}{\mathcal{V}}           %Weyl vertex algebra
\newcommand{\Weylalg}{\mathscr{W}}          %Weyl algebra
\newcommand{\Lieghost}{\mathcal{L}}         %bosonic ghost Lie algebra

\newcommand{\tres}{\mathrm{res}}

\DeclareMathOperator{\id}{Id}

\DeclareMathOperator{\tr}{tr}
\DeclareMathOperator{\Mat}{Mat}
\DeclareMathOperator{\gr}{gr}
\DeclareMathOperator{\Span}{span}
\DeclareMathOperator{\Ima}{Im}
\DeclareMathOperator{\Soc}{Soc}
\DeclareMathOperator{\Rad}{Rad}
\DeclareMathOperator{\ad}{ad}

\DeclarePairedDelimiter{\no}{:}{:}          %Normaly ordered product

\DeclareRobustCommand\longtwoheadrightarrow{\relbar\joinrel\twoheadrightarrow}

\author[Katrina Barron]{Katrina Barron}
\address[K.B.]{Department of Mathematics, University of Notre Dame, Notre Dame, IN 46556}
\email{kbarron@nd.edu}

\author[Justine Fasquel]{Justine Fasquel}
\address[J.F.]{Université Bourgogne Europe, CNRS, IMB UMR 5584, 21000 Dijon, France}
\email{justine.fasquel@u-bourgogne.fr}

\author[Florencia Orosz Hunziker]{Florencia Orosz Hunziker}
\address[F.O.-H.]{Department of Mathematics, University of Colorado Boulder, Boulder, CO 80309}
\email{florencia.orosz@colorado.edu}

\author[Gaywalee Yamskulna]{Gaywalee Yamskulna}
\address[G.Y.]{Department of Mathematics, Illinois State University, Normal, IL 61790}
\email{gyamsku@ilstu.edu}

\thanks{K.B. was supported by a Simons Foundation Travel Support Grant.
J.F.'s research was supported by a University of Melbourne Establishment Grant and an Andrew Sisson Support Package 2025. F.O.H. was partially supported by the US National Science Foundation under Grant No. DMS-2102786.
G.Y. was supported by an AMS-Simons PUI Faculty Grant and a grant from the College of Arts and Sciences at Illinois State University.}

\title{On mode transition algebras  for $\mathbb{Z}$-graded vertex algebras and applications to bosonic ghosts} 

\begin{document}

\begin{abstract}
We study the mode transition algebras and Zhu algebras in the setting of $\mathbb{Z}$-graded vertex algebras, with particular focus on the Weyl vertex algebra $\WeylVA$ at central charge 2 (also known as bosonic ghosts or the $\beta \gamma$-system).
We show that the mode transition algebras of $\WeylVA$ admit unity elements that form a family of strong unities in the sense of Damiolini-Gibney-Krashen. 
The existence of unities for the mode transition algebra of $\WeylVA$ allows us to explicitly construct all higher level Zhu algebras of $\WeylVA$. 
We further analyze weak $\WeylVA$-modules induced from Zhu algebras, proving that every such module is already induced from the level‑zero Zhu algebra.  
We then prove that all indecomposable reducible weight modules induced from a Zhu algebra are not weakly interlocked, and hence not strongly interlocked in the sense of Barron-Batistelli-Orosz Hunziker-Yamskulna.  
More generally, we show that the property of being weakly interlocked is preserved under the action of an invertible Li's $\mathbf{\Delta}$ operator. As an application, we prove that all indecomposable reducible weight $\WeylVA$-modules obtained via spectral flow of Zhu‑induced modules are likewise not weakly interlocked. These results clarify the role of being weakly interlocked in the modularity properties of bosonic ghost modules previously studied by Ridout-Wood and Allen-Wood.
\end{abstract}

\maketitle

\section{Introduction and preliminaries}

Vertex algebras and their modules are the basic building blocks of conformal field theory.
Among the tools developed to study their representation theory, Zhu algebras provide a powerful way to construct and analyze a certain category of modules thanks to the famous Zhu's correspondence \cite{Zhu,FZ,DLM}.
In another direction, the mode transition algebras, introduced more recently \cite{DGK}, offer new insights into both algebraic and geometric aspects of vertex algebras.
In this paper, we study these two families of associative algebras in the setting of $\mathbb{Z}$-graded vertex algebras,  which need not be vertex operator algebras. 

Zhu algebras were originally defined in \cite{Zhu} at level zero and later generalized in \cite{DLM} to higher levels $n \in \mathbb{Z}_{>0}$.
Concretely, they were defined as quotients of a vertex operator algebra $V$. They encode essential information about the representation theory of a vertex algebra, but explicit calculations --- especially at higher levels --- are notoriously difficult.
More recently, He \cite{He} provided a realization of Zhu algebras as quotients of universal enveloping algebras. 
We show that both these equivalent constructions naturally extend from the setting of vertex operator algebras to the setting of $\mathbb{Z}$-graded vertex algebras.

The second family of algebras that we discuss is the family of mode transition algebras.
They were introduced in \cite{DGK} for vertex operator algebras of so-called ``CFT" type, where they revealed deep connections between algebraic properties of the vertex operator algebra $V$ as well as geometric constructions of moduli of curves labeled by certain $V$-modules. 
In the Appendix of \cite{DGK}, other settings for mode transition algebras are given. Here
we show that the construction still holds for $\Z$‑graded vertex algebras (not necessarily with a conformal element),  and that they again interact closely with Zhu algebras in this broader setting.

As an application, we calculate the mode transition algebras for the Weyl vertex algebra at central charge 2, denoted $\WeylVA$,
and also known as bosonic ghosts or the $\beta \gamma$-system. 
This vertex algebra is conformal and $\mathbb{Z}$-graded but it is not a vertex operator algebra.  
We prove that these algebras admit unity elements forming a strong family of unities in the sense of \cite{DGK}.
The fact that these algebras have unities allows us to determine all higher level Zhu algebras for $\WeylVA$. 
This is only the second example in which Zhu algebras have been explicitly described at all levels, the first being the Heisenberg vertex operator algebra where all Zhu algebras were determined in \cite{DGK} using the mode transition algebra, proving a conjecture in \cite{AB23}.

Next, we turn our attention to certain categories of modules for $\mathbb{Z}_{\geq 0}$-graded vertex algebras and their properties.
In particular, the induction functor of \cite{Zhu} and \cite{DLM} that produces a $\mathbb{Z}_{\geq 0}$-gradable $V$-module from a module for a Zhu algebra when $V$ is a vertex operator algebra, is extended here to the setting of $\mathbb{Z}_{\geq 0}$-gradable vertex algebras.  
We then prove two results about the structure of modules for a $\mathbb{Z}_{\geq 0}$-graded vertex algebra that are directly induced from a module for a Zhu algebra or obtained after ``twisting" the first ones by a certain invertible operator. One important property of indecomposable reducible modules is whether they are weakly interlocked, a condition necessary for the existence of well-defined graded pseudo‑traces in the sense of, e.g., \cite{Miy1} or \cite{BBOHY}, and their doubly‑graded analogues in the context of vertex algebras with infinite-dimensional $L_0$-graded spaces but finite-dimensional doubly-graded spaces with respect to another weight-zero operator.  
Graded pseudo-traces are certain $q$-series (or series in two formal variables in the case of doubly-graded finite-dimensional subspaces) associated to a module. 
We prove that modules for a Zhu algebra that fail to be weakly interlocked remain non‑weakly interlocked after induction as vertex algebra modules.

Although the Zhu algebras can be used to induce $\mathbb{Z}_{\geq 0}$-gradable $V$-modules, these modules are generally not enough to form a ``nice" category that has closure under fusion for instance, and/or has certain number-theoretic invariance properties under the action of the modular group $SL_2(\Z)$ on the spaces of graded traces and graded pseudo-traces. 
Twisting the module action by another action is therefore a useful way for constructing new modules --- specifically, ones that cannot be induced from a Zhu algebra --- by acting on modules induced from Zhu algebras. In particular, the action can be twisted by an operator ${\bf \Delta}$ defined by Li in \cite{Li97} that satisfies important properties related to what are called ``simple currents". The most famous ${\bf \Delta}$-operators are the so-called ``spectral flows". 
Simple currents and spectral flow have a long history in conformal field theory due to the relation between spectral flow and orbifold conformal field theory, especially in superconformal settings, cf. \cite{SS, DLM1, MO, LMRS, R, Barron-varna, CR, ACKR, ACK}. 

In this paper, we prove that the property of being weakly interlocked is preserved when twisting the action by an invertible Li's ${\bf \Delta}$-operator. We apply this result to certain modules for the Weyl vertex algebra $\WeylVA$.
Letting $\mathscr{C}_Z$ denote all $\WeylVA$-algebra modules induced from a weight module for a Zhu algebra, we prove that all $\WeylVA$-modules in $\mathscr{C}_Z$ are induced at level zero. 
Moreover, none of the indecomposable reducible $\WeylVA$-modules in $\mathscr{C}_Z$ nor their spectrally flowed images are weakly interlocked. This essentially recovers results previously contained in work by Ridout-Wood \cite{RW} and Allen-Wood \cite{AW}, studying the subcategory $\mathscr{R}$ of $\mathscr{C}_Z$ consisting of weight $\WeylVA$-modules induced from the level zero Zhu algebra but with conformal weights restricted to be real (for an emphasis on physical applications).  
Spectral flows of these modules along with their finite length extensions were studied in \cite{RW, AW} in what they called category $\mathscr{F}$.   
 In \cite{AW}, it is shown that the $\WeylVA$-module category $\mathscr{R}$ is not closed under fusion but that the larger category $\mathscr{F}$ is.
The lack of weakly interlocked reducible weight modules for the Weyl vertex algebras supports the modularity properties studied by Ridout-Wood, i.e., modularity in this setting only relies on doubly-graded traces and not on the notion of doubly-graded pseudo-traces which requires at least weakly interlocked modules in order to be well defined.  

\subsection{Organization and Main Results}

The paper is organized as follows. After introducing the basic definitions, we discuss higher level Zhu algebras and mode transition algebras for $\mathbb{Z}$-graded vertex algebras in Section \ref{MTA-Zhu-section}, showing that the notions and certain results from \cite{Zhu,DLM,He,DGK} still hold in a more general setting than that of vertex operator algebras. 
The main results of Section \ref{MTA-Zhu-section} include Theorem \ref{thm:unity} which gives the structure of the level $d$ Zhu algebra in terms of the $d$th level mode transition algebra and the level $d-1$ Zhu algebra for a $\mathbb{Z}$-graded vertex algebra when the $d$th level mode transition algebra admits a unity. This extends a result of \cite{DGK}.
We also discuss in Theorem \ref{thm:unity} families of strong unities for mode transition algebra in the setting of $\mathbb{Z}$-graded vertex algebras following \cite{DGK}.

In Section \ref{sec:WeylVA}, we start recalling data about the Weyl vertex algebra $\WeylVA$ at central charge $2$.
Then we calculate the level $d$ mode transition algebras $\MTA_d(\WeylVA)$ for $d \in \mathbb{Z}_{\geq0}$ (see Proposition \ref{mode-transition}) and we show that $\MTA_d(\WeylVA)$ admits a unit for each $d$. This allows us to describe all the higher level Zhu algebras for $\WeylVA$ (Corollary \ref{Zhu-alg}). We also show that these unities form a family of strong unities (Theorem \ref{weyl-unity-thm}).  

We turn our attention to modules for $\mathbb{Z}$-graded vertex algebras in Section \ref{modules-section}. 
After some basic definitions, we discuss modules induced from the level $d$ Zhu algebra of a $\mathbb{Z}_{\geq 0}$-graded vertex algebra $V$ (Section \ref{subsec:zhu_correspondence})  and new modules that can be obtained by applying Li's ${\bf \Delta}$-operators such as  spectral flow (Section \ref{subsec:Delta_op}). In Section \ref{subsec:weakly_interlocked}, we introduce the notion of weakly interlocked modules for an associative algebra and for a vertex algebra and obtain two noticeable results. First, we prove that the Zhu induction of a non-weakly interlocked module for a Zhu algebra gives rise to a non-weakly interlocked $V$-module for the corresponding vertex algebra (Theorem \ref{Zhu-weakly-interlocked-thm}). Next, we show that if $W$ is a weak $V$-module and ${\bf \Delta}$ is invertible, then $W$ is weakly interlocked if and only if the corresponding ${\bf \Delta}$-twisted module is weakly interlocked (Theorem \ref{spectral-flow-thm}).

Finally, in Section \ref{Weyl-modules-first-section}, we apply Theorems \ref{Zhu-weakly-interlocked-thm} and \ref{spectral-flow-thm} to the Weyl vertex algebra $\WeylVA$ at central charge $2$, to prove that no indecomposable reducible weight $\WeylVA$-modules constructed via spectral flow of a $\WeylVA$-module are weakly interlocked.

\subsection{Acknowledgements} 
The authors gratefully acknowledge the support provided by a Simons Foundation Travel Support Grant, a University of Melbourne Establishment Grant and an Andrew Sisson Support Package (2025), the U.S.\ National Science Foundation Grant, an AMS--Simons PUI Faculty Grant, and a grant from the College of Arts and Sciences at Illinois State University, as well as the hospitality of the University of Colorado Boulder and the Institut de Math\'ematiques de Bourgogne. This project was initiated through the Women in Mathematical Physics program, and the authors are sincerely grateful to Ana Ros Camacho for her dedication to organizing the program.

\subsection{Preliminaries}
 We start by recalling some properties of vertex algebras. We refer the reader, for instance, to \cite{LL,LY} for additional details.

\begin{defn}\phantom{x}
\begin{enumerate}[leftmargin=*,label=(\roman*)]
    \item A \emph{$\Z$-graded vertex algebra} consists of a vertex algebra $V$ with a grading
\[V=\bigoplus_{m \in \mathbb{Z}}V_{m} \]
satisfying $\vac \in V_0$ and 
\begin{align*}
u_{n} V_{m}\subset V_{k+m-n-1},    
\end{align*}
for $u \in V_k$, 
where $Y(u,z) = \sum_{n \in \Z} u_n z^{-n-1}$.

\item If there exists $M \in \Z_{\leq 0}$ such that $V_m = 0$ for $m < M$, we call $V$ a \emph{$\Z_{\geq 0}$-gradable vertex algebra},  where $\Z_{\geq 0} = \{ n \in \mathbb{Z} \; | \; n \geq 0 \}$.
\end{enumerate}
\end{defn}

\begin{rem}\phantom{x}
\begin{enumerate}[label=(\roman*),leftmargin=*]
    \item Any vertex algebra $V = (V, Y, {\bf 1})$ is endowed with the linear translation operator $T: V\to V$ determined by $T (u)=u_{-2}\vac$.
    \item For any $\Z$-graded vertex algebra $V$, we can define the \emph{grading operator} $\Delta$ as the linear map $\Delta: V\to V$ determined by $\Delta (u)=ku$ for $u\in V_k$.
    An element $u\in V_k$ is said to be homogeneous and we set $\wt (u)=k$. When the notation $\wt(u)$ is used, then $u$ is assumed to be homogeneous.
    \item We used the term ``$\Z_{\geq 0}$-gradable" instead of ``$\Z_{\geq 0}$-graded" above to stress
    that if $V$ has finitely many negative graded spaces, a shift can be used to give $V$ a $\Z_{\geq 0}$-grading (in which there are no negative weight spaces). More precisely, if $V=\bigoplus_{m\in \Z_{\geq M}}V_m$, setting $V(j):=V_{M+j}$ for $j\geq 0$ yields to the $\Z_{\geq 0}$-grading 
\begin{equation*}
    V=\bigoplus_{j\in \Z_{\geq 0}}V(j).
\end{equation*}
\end{enumerate}
\end{rem}

\begin{defn} 
A \emph{$\Z$-graded conformal vertex algebra} $(V, Y, \vac, \omega)$ consists of a $\Z$-graded vertex algebra
together with a distinguished vector $\omega \in V_2$ whose modes $Y(\omega, z)=\sum_{n\in \mathbb{Z}} L_{n} z^{-n-2}$
satisfy the following properties: 
\begin{enumerate}[label=(\roman*),leftmargin=*]
\item the Virasoro relations: for any $n,m\in \Z$,
$$[L_{n},L_{m}]=(n-m)L_{m+n}+\frac{1}{12}(n^3-n)\delta_{n,-m}c,$$ where $c\in \mathbb{C}$ is the \emph{central charge} of $V$, 
    \item the $L_{-1}$-derivative property: for any $v\in V$,
			$Y(L_{-1}v,z)=\dfrac{d}{dz} Y(v,z),$
   \item the $L_{0}$-grading property: for $m \in \Z$ and $v\in V_{m}$,
			    $L_{0}v=m v=(\wt(v))v$, and we call $m = \wt(v)$ the {\it conformal weight of $v$}.
\end{enumerate}
 If in addition $(V,Y, {\bf 1} , \omega)$ is $\Z_{\geq 0}$-gradable with respect to the $L
_0$-grading, we call $V$ a \emph{$\Z_{\geq 0}$-gradable conformal vertex algebra} (still with possibly negative conformal weight spaces).
\end{defn}

\begin{defn} A \emph{vertex operator algebra} $(V,Y,\vac,\omega)$ is a  $\Z_{\geq 0}$-gradable conformal vertex algebra 
\[V=\bigoplus_{m \in \mathbb{Z}}V_{m}\]
such that $\dim V_{m}<\infty$ for $m\in \Z.$
\end{defn}

\section{Higher level Zhu algebras and mode transition algebras for \texorpdfstring{$\Z$}{Z}-graded vertex algebras}\label{MTA-Zhu-section}

In this section we recall various algebras related to a vertex operator algebra, and recall (or define when necessary) their extensions to $\mathbb{Z}$-graded vertex algebras.  

 We first recall the notions of the Zhu algebras as defined originally in \cite{Zhu} and \cite{DLM} for vertex operator algebras, and their realization given  in  \cite{He} as quotients of a certain universal enveloping algebra. We show that these notions extend naturally to the setting of $\mathbb{Z}$-graded vertex algebras. Then following \cite{DGK}, we recall the notion of the mode transition algebra in the context of $\mathbb{Z}$-graded vertex algebras and explain its relation to the higher level Zhu algebras.

Throughout this section, we fix $V = \coprod_{n \in \mathbb{Z}}V_n$ to be a $\mathbb{Z}$-graded vertex algebra with possibly infinite-dimensional graded spaces. 

\subsection{Zhu algebras associated to  a \texorpdfstring{$\Z$}{Z}-graded vertex algebra} 
In this subsection, we recall the definition of the level $n$ Zhu algebra ($n \in \Z_{\geq 0}$), denoted by $\hZhu{n}(V)$.
The level $0$ Zhu algebra, simply called the Zhu algebra, was introduced by Zhu in \cite{Zhu} for $\mathbb{Z}_{\geq 0}$-graded vertex operator algebras. 
Then, the definition was generalized by Dong--Li--Mason to the level $n$ Zhu algebra for $n >0$ in \cite{DLM} for any vertex operator algebra.  
In these works, it was shown that if $V$ is a vertex operator algebra, then $\hZhu{n}(V)$ can be endowed with a structure of unital associative algebra.
Although in both \cite{Zhu} and \cite{DLM}, $V$ is assumed to be a vertex operator algebra, the construction of $\hZhu{n}(V)$ can be extended beyond this case,  as discussed for instance in  \cite{KDS,Eke11,EH19}. In this subsection, we consider the construction for general $\Z$-graded vertex algebras,  that is, without any restrictions on the number of graded weight spaces with negative conformal weight or on the dimension of their weight spaces.

\begin{defn}\label{zhu-alg-defnition}
Let $V=\bigoplus_{m\in\Z}V_m$ be a $\mathbb{Z}$-graded vertex algebra. 
For $n \in \mathbb{Z}_{\geq 0}$, let $O_n(V)$ be the subspace of $V$ spanned by elements of the form
\begin{equation}\label{elements-in-O} 
u \circ_n v :=
\res_x \frac{(1 + x)^{\wt(u) + n}Y(u, x)v}{x^{2n+2}},\qquad u,v\in V,
\end{equation}
and by elements of the form 
\begin{equation}\label{eq:derivation_relation}
    (T + \Delta)v,\qquad v\in V,
\end{equation}
where $T$ and $\Delta$ are respectively the translation and the grading operators on $V$. 
The vector space $\hZhu{n}(V)$ is defined to be the quotient space $V/O_n(V)$.
\end{defn}

For $n\geq0$, consider the following operation $ *_n $ on $V$ defined by linearity with
\[
u *_n v = \sum_{m=0}^n(-1)^m\binom{m+n}{n}\res_x \frac{(1 + x)^{\wt(u) + n}Y(u, x)v}{x^{n+m+1}}, \qquad u,v\in V.
\]

\begin{prop}[{\cite{Zhu,DLM,KDS}}]  If $V$ is a $\mathbb{Z}$-graded vertex algebra, then for $n \in \mathbb{Z}_{\geq 0}$, the space $O_n(V)$ is a two-sided ideal of $V$ with respect to $*_n$, and the quotient $\hZhu{n}(V) = V/O_n(V)$ is a unital associative algebra  with respect to $*_n$ and ${\bf 1}$, called the \emph{level $n$ Zhu algebra}.
\end{prop}

When $V = (V, Y, {\bf 1}, \omega)$ is a vertex operator algebra with non-negative weights, then, 
$\hZhu{0}(V)$ was first defined in \cite{Zhu} and, for any vertex operator algebra,  $\hZhu{n}(V)$ for $n\geq 0$ was defined in  \cite{DLM}, with $T = L_{-1} = \omega_0$ and $\Delta = L_{0} = \omega_{1}$.  
In \cite{KDS}, it was noticed that this definition can be extended to any vertex algebra which is $\Z$-graded by a Hamiltonian, for instance by replacing finite polynomials with Taylor expansions when necessary. Indeed the definition of $\Zhu_n(V)$ for $V$ a vertex operator algebra as given in \cite{DLM} is still valid if one replaces $L_{-1}$ by the translation operator and $L_{0}$ by the grading operator as in the definition of a $\mathbb{Z}$-graded vertex algebra. 
For instance, in \cite{L1}, the (level zero) Zhu algebra for the Weyl conformal vertex algebra at central charge $c = 2$ (see Subsection \ref{subsectionWeyl} below) was calculated, using the fact that the notion of Zhu algebra is well defined for a $\mathbb{Z}_{\geq 0}$-gradable conformal vertex algebra that is not a vertex operator algebra. 
We refer to \cite{BBOHPTY}  for the reader interested in results on the Zhu algebra of the Weyl vertex algebra for other choices of conformal element. 

\begin{rem}\label{rem:virasoro_rel_in_zhu}
When $V$ is a vertex operator algebra, since 
$$v \circ_0 \mathbf{1} = v_{-2} \mathbf{1} + (\wt( v)) v = L_{-1} v + L_{0} v,$$
it follows that $O_0(V)$ is spanned by elements of the form \eqref{elements-in-O} and $\hZhu{0}(V)$ coincides with the original definition of the Zhu algebra $\Zhu(V)$ \cite{Zhu}.
This is not necessarily true anymore for $n>0$.
In particular, in \cite{AB-generaln}, specific instances for when elements of the form $L_{-1}v + L_{0}v$ can not be written as a linear combination of elements of the form \eqref{elements-in-O} are given.
These cases give new relations in $\hZhu{n}(V)$ that are not captured by the relations arising from the span of elements of the form \eqref{elements-in-O}.  The remark also holds if $V$ is a $\Z$-graded vertex algebra where the elements of the form \eqref{eq:derivation_relation} are not necessarily in the space spanned by elements of the form \eqref{elements-in-O}. 
\end{rem}

\subsection{The Lie algebra of a \texorpdfstring{$\Z$}{Z}-graded vertex algebra, filtrations, universal enveloping algebras, and completions}\label{completion-subsection}

Following \cite{DGK}, we discuss various filtrations of interest for studying higher-level Zhu algebras.  We will use the following notations:
\begin{eqnarray*}
\mathbb{C}[[t,t^{-1}]] &=& \bigg\{ \sum_{n\in\mathbb{Z}} a_n t^n \; | \; a_n \in \mathbb{C} \bigg\},\\
\mathbb{C}\pp{t} &=& \bigg\{ \sum_{n \in \mathbb{Z}} a_n t^n \, | \, a_n = 0 \mbox{ for all but a finite number of $n<0$} \bigg\}\\
\mathbb{C}[t, t^{-1}] &=& \bigg\{ \sum_{n \in \mathbb{Z}} a_n t^n \, | \, a_n = 0 \mbox{ for all but a finite number of $n$}\bigg\}.
\end{eqnarray*}
We call the first ring {\it formal series in $t$ and $t^{-1}$}, the second ring {\it formal Laurent series in $t$,} and the last set {\it formal Laurent polynomials in $t$}.

First note that the space of Laurent polynomials $\mathbb{C}[t, t^{-1}]$ is $\mathbb{Z}$-graded by $\mathbb{C}[t,t^{-1}]_n=\mathbb{C}t^{-n-1}$.
Thus, the space of Laurent series $\mathbb{C}\pp{t}$ has a left (i.e., increasing) filtration with respect to this grading, given by $\mathbb{C}\pp{t}_{\leq n}= t^{-n-1} \mathbb{C}[[t]]$.
This filtration is exhaustive and is a split filtration in the sense of \cite[Appendix A]{DGK}. 
Similarly, there is a right (i.e., decreasing) exhaustive split filtration of $\mathbb{C}\pp{t^{-1}}$ given by $\mathbb{C}\pp{t^{-1}}_{\geq n} = t^{-n-1}\mathbb{C}\bb{t^{-1}}$.  
Of course $\mathbb{C}[t,t^{-1}]$ is both left and right exhaustively filtered as a subring of these left and right filtrations of $\mathbb{C}\pp{t}$ and $\mathbb{C}\pp{t^{-1}}$ above.  However, we note that $\mathbb{C}\bb{t, t^{-1}}$ does not admit an exhaustive left or right filtration that restricts to either of these filtrations, respectively. 

Let $V$ be a $\mathbb{Z}$-graded vertex algebra.  Then $V$ has left and right exhautive filtrations given by $V_{\leq n} = \coprod_{k \leq n} V_k$ and $V_{\geq n} = \coprod_{k \geq n} V_k$, respectively.  
Moreover, the space $V \otimes_{\mathbb{C}} \mathbb{C}\pp{t}$ admits a left filtration with respect to the left filtration of the tensor factors.  
In particular, this filtration endows the algebra with a structure of graded algebra with respect to the \emph{degree} grading of $v \otimes t^n$ defined by
\begin{equation}\label{degree} 
\deg (v \otimes t^n) = \wt(v) -n-1.
\end{equation}
More precisely, the associated graded space is defined by
\[
\gr \, (V \otimes \mathbb{C}\pp{t} ) = \bigoplus_{n \in \mathbb{Z}} (V \otimes \mathbb{C}\pp{t}_{\leq n}) / (V \otimes \mathbb{C}\pp{t}_{\leq n - 1})
\]
and then we have the graded subspaces 
\[ (V \otimes \mathbb{C}\pp{t})_n = \Span_\mathbb{C} \{ v \otimes t^{\wt(v)-n - 1}  \; | \; v \in V \, \text{homogeneous} \}\subset V \otimes \mathbb{C}[t,t^{-1}].\]
This implies that the left filtration on $V\otimes \mathbb{C}\pp{t}$ is a split filtration in the sense of \cite{DGK} with respect to the graded subspace $V \otimes \mathbb{C}[t,t^{-1}]$.  Similarly we can define the graded structure on $V \otimes \mathbb{C}\pp{t^{-1}}$ which is also a split right filtration, again with respect to the graded subspace $V \otimes \mathbb{C}[t,t^{-1}]$.  

Define the linear endomorphism of degree one 
\begin{equation*}
    D=T\otimes \id+\id \otimes \frac{d}{dt}
\end{equation*} 
on $V \otimes \mathbb{C}\pp{t}$, $V \otimes \mathbb{C}\pp{t^{-1}}$ and $V\otimes_{\mathbb{C}}\mathbb{C}[t,t^{-1}]$ respectively
and consider the three quotient spaces
\begin{equation*}%\label{define-Ls}
\mathfrak{L}(V)^L=\frac{V\otimes_{\mathbb{C}}\mathbb{C}\pp{t}}{\Ima(D)}, \quad \mathfrak{L}(V)^R=\frac{V\otimes_{\mathbb{C}}\mathbb{C}\pp{t^{-1}}}{\Ima(D)}, \quad \mathfrak{L}(V)^f=\frac{V\otimes_{\mathbb{C}}\mathbb{C}[t,t^{-1}]}{\Ima(D)}.
\end{equation*}
They are left split-filtered, right split-filtered, and graded, respectively. 
Moreover, they admit filtered and graded Lie algebra structures with respect to the Lie bracket
\begin{equation}\label{commutator-larger}  
[u \otimes f(t) ,v\otimes g(t)]=\sum_{j\geq 0}\frac{1}{j!}(u_jv) \otimes g(t) \frac{d^j}{dt^j} f(t),
\end{equation}
for all $u,v\in V$, and $f(t), g(t) \in\mathbb{C}\pp{t}$ or $\mathbb{C}\pp{t^{-1}}$.
Note that the space $\mathfrak{L}(V)^f$ is usually denoted $\hat{V}$ in the setting of vertex operator algebras (eg. \cite{DLM, He}).  

For $v\in V$ and $m\in\Z$, denote by $v(m)$ the image of $v\otimes t^m$ in $\mathfrak{L}(V)^f$, i.e.
\[v(m) = v \otimes t^m + D(V \otimes \mathbb{C}[t,t^{-1}]).\]
Then the Lie bracket \eqref{commutator-larger} restricted to $\mathfrak{L}(V)^f$ simplifies as
\begin{equation*}%\label{commutator}  
[u(m),v(n)]=\sum_{j\geq 0}{\binom{m}{j}}(u_jv)(m+n-j),\qquad u,v\in V,\ m,n\in\Z.
\end{equation*}

Continuing to follow \cite[Appendix A]{DGK} extended to the $\mathbb{Z}$-graded vertex algebra setting, we let $U^L$, $U^R,$ $U$ denote the universal enveloping algebras of $\mathfrak{L}(V)^L$,  $\mathfrak{L}(V)^R$, $\mathfrak{L}(V)^f$, respectively. Since these enveloping algebras are left filtered, right filtered, and graded respectively, $(U^L, U^R, U)$ forms a good triple of associative algebras and we can define canonically a good seminorm associated with the system of neighborhoods
\[
\NL^n U^L = U^L U^L_{\leq -n},  
%\quad \NL^n U = U U_{\leq -n}, 
\qquad \NR^n U^R = U^R_{\geq n}U^R, 
%\quad \NR^n U = U_{\geq n}U,
\qquad n \in \mathbb{Z},
\]
 that is, $\NL$ and $\NR$ define almost canonical split-filtered seminorms on $U^L$ and $U^R$ respectively such that $\NL^n U_p=\NR^{n+p} U_p$.
As left ideals of $U$ \cite[Lemma 2.4.2]{DGK}, 
\[
\NL^n U^L = U^L \mathfrak{L}(V)^L_{\leq -n},
\qquad \NL^n U := UU_{\leq -n}= U\mathfrak{L}(V)^f_{\leq -n},
\]
while as right ideals
\[
\NR^nU^R = \mathfrak{L}(V)^R_{\geq n}U^R, 
\qquad \NR^nU := U_{\geq n}U= \mathfrak{L}(V)^f_{\geq n}U,
\]
for all $n \in \mathbb{Z}$. 
Taking the completion with respect to this system of neighborhoods, we obtain the completions $\hat{U}^L$, $\hat{U}^R$, and $\hat{U}$, respectively, and $(\hat{U}^L, \hat{U}, \hat{U}^R)$ is again a good triple of associative algebras with good seminorms (see \cite[Corollary A.9.10]{DGK}). 

The universal enveloping algebra $U=\bigoplus_{n\in\Z}U_n$ inherits a natural $\Z$-grading from the Lie algebra $\mathfrak{L}(V)^f$ defined by the degree as in Eqn. \eqref{degree}. 
The zero component $U_0$ contains the universal enveloping algebra  $U(0)$  of the Lie subalgebra $\mathfrak{L}(V)^f(0)$ of $\mathfrak{L}(V)^f$ spanned by homogeneous vectors of zero degree. 
For $n\in\mathbb{Z}$ and $k\in\mathbb{Z}_{\leq 0}$, let 
\begin{equation}\label{eq:filtration}
U^k_n:= \NL^{-k}U_{n} = \sum_{i\leq k} U_{n-i}U_i\quad\text{and}\quad U(0)^k:=U(0)\cap U^k_0,
\end{equation}
which define increasing  filtrations (with respect to $k$) on $U_n$ and $U(0)$ respectively.

\subsection{The left, right, and finite universal enveloping algebras of a \texorpdfstring{$\Z$}{Z}-graded vertex algebra}\label{universal-lie-section}

Let $V$ be a $\mathbb{Z}$-graded vertex algebra, and let $(\hat{U}^L, \hat{U}, \hat{U}^R)$ be the good triple of completions of the universal enveloping algebras of the Lie algebras $\mathfrak{L}(V)^L, \mathfrak{L}(V)^f, \mathfrak{L}(V)^R$ constructed in Section \ref{completion-subsection}, along with good seminorms.

Continuing to follow and extend \cite[Appendix A]{DGK}, as well as \cite{He}, we let $\mathscr{J}^L, \mathscr{J}$, and $\mathscr{J}^R$ be the two-sided graded ideals of $\hat{U}^L, \hat{U}$ and $\hat{U}^R$ respectively, generated by 
\begin{equation}\label{vacuum}
{\bf 1}(-1) = 1_{\hat{U}}  
\end{equation}
and 
\begin{multline}\label{Jacobi-relations} 
    \sum_{i \geq 0} (-1)^i \binom{\ell}{i} \left(u(m+\ell-i)v(n+i)-(-1)^\ell v(n+\ell-i)u(m+i)\right)\\
=\sum_{i \geq 0} \binom{m}{i} (u_{\ell+i}v)(m + n- i),
\end{multline}
for all $u,v\in V$, $m,n,\ell\in\Z$.
Then $(\mathscr{J}^L, \mathscr{J}, \mathscr{J}^R)$ is a good triple of ideals by \cite[Lemma A.9.5]{DGK}.  
Furthermore, since it arises from a canonical seminorm, the seminorm on the good triple $(\hat{U}^L, \hat{U}, \hat{U}^R)$ is \emph{tight} \cite[Remark A.6.11]{DGK} and thus, the respective closure of $\mathscr{J}^L, \mathscr{J}, \mathscr{J}^R$ with respect to this tight seminorm, give a good triple of ideals, denoted $(\overline{\mathscr{J}}^L, \overline{\mathscr{J}}, \overline{\mathscr{J}}^R)$ \cite[Lemma A.9.6]{DGK}.

By \cite[Corollary A.9.9]{DGK}, the quotient spaces \[ \mcU^L = \hat{U}^L/\overline{\mathscr{J}}^L, \quad \mcU = \hat{U}/\overline{\mathscr{J}}, \quad \mcU^R = \hat{U}^R/\overline{\mathscr{J}}^R\]
form a good triple of associative algebras with good seminorms.  We call these associative algebras the \emph{left, finite} and \emph{right universal enveloping algebras of $V$}, respectively.

\begin{rem}
For $\mathbb{Z}_{\geq 0}$-graded vertex operator algebras, $\mcU$ coincides  
with the universal enveloping algebra of $V$ introduced by Frenkel and Zhu \cite{FZ} -- and denoted $U(V)$ in \cite{He}, while as noted in \cite{DGK}, we can identify $\mcU^L$ with the current algebra introduced by Nagatomo and Tsuchiya \cite{NT} or with the universal enveloping algebra $\widetilde{\mathscr{U}}(V)$ introduced in \cite{FLanglands} with a minor modification. 
\end{rem}
It is helpful to visualize the following maps, where the second and the third maps are embedding and projection, respectively,
\begin{align*}
V&\longrightarrow  \ \  \mathfrak{L}(V)^f &&\longhookrightarrow \hat{U}  &&  \longtwoheadrightarrow\ \mcU \\
 v &\longmapsto \ v(\wt(v)-1) &&\longmapsto (v(\wt(v)-1)).1  &&\longmapsto  (v(\wt(v) - 1)).1 +\overline{\mathscr{J}} \\
 \vac &  \longmapsto \ \vac(-1) &&\longmapsto  \vac(-1) &&\longmapsto  1 + \overline{\mathscr{J}}\\
& \ \ \ \ \ \ \ \   {\bf 1}(j) = \delta_{j, -1}{\bf 1}(-1)  && \longmapsto \delta_{j, -1}{\bf 1}(-1) && \longmapsto \delta_{j, -1} + \overline{\mathscr{J}}. 
\end{align*}

\begin{rem}
In \cite{He} (in the context of vertex operator algebras) and other places in the literature, there are three families of relations defining $\mathscr{J}$.  
However, the two relations \eqref{vacuum} and \eqref{Jacobi-relations} are sufficient. In particular, the relation that is often given
\begin{equation}\label{he-vaccum}
{\bf 1} (j) =  {\bf 1}\otimes t^j 
+ D(V \otimes \mathbb{C}[t, t^{-1}])= 0, \qquad j\neq-1,
\end{equation}
is already a relation in $\hat{U}$ since 
\begin{equation*}
D ({\bf 1}(j+1)) = T({\bf 1})\otimes t^{j+1}+{\bf 1}\otimes \frac{d}{dt}t^{j+1} =  (j+1){\bf 1}\otimes t^{j} \equiv 0. 
\end{equation*}

Another family of relations that is often specified in the literature is the Virasoro relations if $V$ has a conformal structure.  But again, if present in $V$, then these relations are already present in $\mathfrak{L}(V)^f$, and so in $\hat{U}$ and $\mcU$.  
Recall (see for instance \cite{LL}) that the Virasoro relations are equivalent to 
\begin{eqnarray*}
 \omega_0 \omega &=& L_{-1} \omega,\\
 \omega_1 \omega &=& 2L_{-2} {\bf 1} = 2\omega,\\
 \omega_2\omega &=& 0,\\
 \omega_3 \omega &=& \frac{1}{2} c {\bf 1},\\
 \omega_n\omega &=& 0 \quad \mbox{for $n\geq 4$.}
\end{eqnarray*}
Then in $\mathfrak{L}(V)^f$ and $\hat{U}$, we have 
\begin{eqnarray*}
\lefteqn{[\omega(m +1), \omega (n+1)] =\sum_{i\geq 0}{m+1\choose i} (\omega_i\omega)(m + n + 2 -i)} \\
&=&(L_{-1}\omega)(m+n+2)+{m+1\choose 1}2\omega(m+n+1)+{m+1\choose 3}\frac{1}{2} c \ {\bf 1}(m+n-1)\\
&=&-(m+n+2)\omega(m+n+1)+(m+1)2\omega(m+n+1)+{m+1\choose 3}\frac{1}{2} c \delta_{m,-n} \ {\bf 1}(-1)\\
%&=&-(m+n+2)\omega(m+n+1)+(m+1)2\omega(m+n+1)\\
%&& \quad +\frac{(m+1)m(m-1)}{12} c \delta_{m,-n} \ {\bf 1}(-1)\\
&=&(m-n)\omega(m+n+1)+\frac{(m+1)m(m-1)}{12} c \delta_{m,-n} \ {\bf 1}(-1).
\end{eqnarray*}
Thus, using the relation ${\bf 1} (-1) + \mathscr{J} = 1$,  we recover the Virasoro relations in $\mcU$ as well.
\end{rem}

\subsection{Zhu algebras as quotients of the universal enveloping algebra}\label{He-section}

In \cite{He}, He proves that when $V$ is a vertex operator algebra, $\hZhu{n}(V)$ is isomorphic to a quotient of the universal enveloping algebra of $V$. 

In \cite{He}, the universal enveloping algebra $\mcU$ of a vertex operator algebra $V$ is denoted $U(V)$. 
Moreover, He shows that if $V$ is a vertex operator algebra then
    \begin{equation}\label{he-isomorphism}
        \HeZhu{n}(V):=U(V)_0/U({V})^{-n-1}_0\cong \hZhu{n}(V), 
    \end{equation}
for all $n\in\Z_{\geq0}$, i.e., that the quotient algebra $\HeZhu{n}(V)$ is canonically isomorphic to the level $n$ Zhu algebra $\hZhu{n}(V)$\footnote{
  There are typos in the indexing of the modes in \cite[Corollary A.2]{He}, both in the last three lines of the equation in the proof of Corollary A.2 and in the second line of the equation giving the result. Corollary A.2 is used in the proof of Lemma 3.1, which still holds despite the indexing typos, and thus the result given here as Eqn. \eqref{he-isomorphism} still holds for vertex operator algebras}.
With our notation (and the notation of \cite{DGK}), Eqn. \eqref{he-isomorphism} reads 
    \begin{equation}\label{he-isomorphism-DGK}
        \HeZhu{n}(V):=\mcU_0/ \mathrm{N}^{n+1} \mcU_0 \cong \hZhu{n}(V) 
    \end{equation} 
where $\mathrm{N}^{n+1} \mathscr{U}_0 := \NL^{n+1} \mathscr{U}_0 = \NR^{n+1} \mathscr{U}_0$.

This fact was presented in \cite{DGK} for $V$ a vertex operator algebra, where the notion of Zhu algebras was also generalized to $\mathbb{Z}_{\geq 0}$-graded seminormed unital algebras.
An explicit isomorphism between the level $n$ Zhu algebra $\hZhu{n}(V)$ and the zero mode algebra $\HeZhu{n}(V)$ is given by the map
    \begin{eqnarray}\label{he-isomorphism-map}
        \varphi: \hZhu{n}(V) &\longrightarrow& \mcU_0/\mathrm{N}^{n+1} \mcU_0\\
        v + O_n(V) &\longmapsto & v(\wt(v) - 1)+ \mathrm{N}^{n+1} \mcU_0. \nonumber 
    \end{eqnarray}
In particular, one can identify $\hZhu{0}(V)$ with $\HeZhu{0}(V)=\mcU_0/\mathrm{N}^{1} \mcU_0$.

\begin{rem}\label{zero-mode-remark}
 The realization of the (level zero) Zhu algebra as the algebra of zero modes of $V$ was the main motivation behind its definition \cite{Zhu, FZ}. The identification $\hZhu{0}(V)\cong \HeZhu{0}(V)$ has appeared in the literature before He's work in \cite{EG} (see also \cite[Appendix~B]{RW1}).    
\end{rem}

As mentioned in Remark \ref{rem:virasoro_rel_in_zhu}, the original definition of $\hZhu{n}(V)$, for $n>0$, required moding out by the relation $(L_{-1} + L_0)v=0$ in a vertex operator algebra, or equivalently by $(T + \Delta)v = 0$ in a $\mathbb{Z}$-graded vertex algebra. 
Thus, one might worry whether the higher level Zhu algebra realized as the quotient $\HeZhu{n}(V)$ of the universal enveloping algebra of $V$ does contain the analogous relation.
It does, and we check this by  computing the image of  $(T+\Delta)v$ through the explicit isomorphism \eqref{he-isomorphism-map}:
\begin{eqnarray*}
 (T v)(\wt(Tv)-1) + (\Delta v)(\wt(\Delta v)-1) 
&=& (T v)(\wt(v)) + (\Delta v)(\wt(v)-1) \\
&=& (Tv)\otimes t^{\wt(v)}+\wt(v)\, v\otimes t^{\wt(v)-1}\\
&=& -\wt(v)\, v\otimes t^{\wt(v) -1}+\wt(v)\, v\otimes t^{\wt(v)-1} \\
&= & 0 \in \mcU
\end{eqnarray*} 
since we have $T v\otimes t^{\wt(v)}+v\otimes \frac{d}{dt}t^{\wt(v)}=0$ in $\mcU$.
More importantly, this leads to naturally considering the right-hand side of the isomorphism \eqref{he-isomorphism} for vertex algebras without a conformal element.
This framework, for instance, is referenced in \cite{DGK}, where they develop notions of Zhu algebras and mode transition algebras in the general context of $\mathbb{Z}_{\geq 0}$-graded algebras, where the Zhu algebras are studied in the spirit of He's approach, i.e., via algebras of modes of $V$ in the universal enveloping algebra of $V$ rather than as quotient spaces of $V$ as in \cite{Zhu,DLM}.

In the following Proposition, we note that the realization of the level $n$ Zhu algebra via modes of $V$ in the universal enveloping algebra of $V$ can be extended to the case when $V$ is a $\mathbb{Z}$-graded vertex algebra.

\begin{prop}\label{define-Zhu-alg-prop}  
Let $V$ be a $\mathbb{Z}$-graded vertex algebra. Then the construction in this section of $\mcU$ along with the gradings and filtrations  is  well defined. Moreover
 for $n \in\mathbb{Z}_{\geq 0}$, the quotient space 
$\HeZhu{n}(V):=\mcU_0/\mathrm{N}^{n+1} \mcU_0$ is a well-defined associative algebra that is isomorphic to the level $n$ Zhu algebra $\hZhu{n}(V) = V/O_n(V)$.  That is 
\[
\HeZhu{n}(V) = \mcU_0/\mathrm{N}^{n+1} \mcU_0  \cong \hZhu{n}(V) = V/O_n(V).
\]
\end{prop}

\subsection{The mode transition algebra for a \texorpdfstring{$\Z$}{Z}-graded vertex algebra}\label{sec:MTA_Zgraded}

Mode transition algebras are associative algebras introduced in \cite{DGK} for certain vertex operator algebras $V$ that carry information on their higher level Zhu algebras and on geometric aspects of the deformation of sheaves of $V$-modules over families of curves. In \cite{DGK}, the setting of most of the work is for vertex operator algebras of \emph{CFT-type}, which are vertex operator algebras with non-negative weight spaces and a one-dimensional zero weight space.

Importantly, mode transition algebras are more broadly defined for $\mathbb{Z}_{\geq 0}$-graded algebras, as was noted in \cite[Appendix B]{DGK}.
In this Section, we extend the notion to $\mathbb{Z}$-graded vertex algebras mainly by noting that the work of \cite{DGK} carries through in this setting.  The bulk of this work has in fact already been done in Section \ref{completion-subsection}.

Fix a $\mathbb{Z}$-graded vertex algebra $V$ and
define the vector space 
\[
\MTA = \mcU/\NL^1\mcU \otimes_{\mcU_0}\HeZhu{0}\otimes_{\mcU_0}\mcU/\NR^1\mcU,
\]
where, $\mathrm{A}_{ 0}:=\HeZhu{0}(V)$ and $\mcU$ is the (finite) universal enveloping algebra of $V$ constructed in Section \ref{universal-lie-section}. 
This $\MTA$ is the underlying vector space of the \emph{mode transition algebra} which is naturally doubly $\mathbb{Z}$-graded as follows. 
The quotients of $\mcU$ inherit a $\mathbb{Z}$-grading from $\mcU$ (which is derived from the degree grading $\deg v(m) = \wt(v) - m - 1$ on the Lie algebra $\mathfrak{L}(V)^f$ associated to $V$; see Section \ref{completion-subsection}).

Let $\left(\mcU/\NL^1\mcU\right)_{d},  \left(\mcU/\NR^1\mcU\right)_{d}$ be the respective spaces of equivalence classes of degree $d$ elements in $\mcU$. 
Concretely, we have
\begin{equation}\label{eq:quotients_N}
\begin{aligned}
    &\left(\mcU/\NL^1\mcU\right)_{d}:= \mcU_d/\NL^1\mcU_d=\mcU_d/\left(\sum_{j\leq-1}\mcU_{d-j}\mcU_j\right),\\
    &\left(\mcU/\NR^1\mcU\right)_{d}:= \mcU_d/\NR^1\mcU_d=\mcU_d/\left(\sum_{j\geq1}\mcU_{j}\mcU_{d-j}\right).
\end{aligned}
\end{equation}

For $d_1,d_2\in\Z$, we define $\MTA_{d_1,d_2}$ to be the vector space
\begin{equation}\label{eq:MTA}
    \MTA_{d_1,d_2}=\left(\mcU/\NL^1\mcU\right)_{d_1}\otimes_{\mcU_0}\HeZhu{0}\otimes_{\mcU_0}\left(\mcU/\NR^1\mcU\right)_{d_2}.
\end{equation}
Then we have the grading 
\begin{equation*}   \MTA=\bigoplus_{\substack{d_1\in\Z_{\geq0}\\d_2\in\Z_{\leq0}}}  \MTA_{d_1, d_2},
\end{equation*}
with $\MTA_{0,0} = \left(\mcU/\NL^1\mcU\right)_0\otimes_{\mcU_0}\mathrm{A}_ 0\otimes_{\mcU_0}\left(\mcU/\NR^1\mcU\right)_0 \cong  \mathrm{A}_0 \otimes_{\mathrm{A}_0} \mathrm{A}_0\otimes_{\mathrm{A}_0} \mathrm{A}_0 \cong \mathrm{A}_0$.

 We define an algebra structure on $\MTA$ by recalling the multiplicative structure described in \cite{DGK} and note that it extends to the setting of a $\mathbb{Z}$-graded vertex algebra.
For $\mathfrak{a}=u\otimes a\otimes u'$ and $\mathfrak{b}=v\otimes b\otimes v'$ two vectors in $\MTA$, set
\begin{equation}\label{eq:star_product}
    \mathfrak{a}\star\mathfrak{b}=u\otimes a(u'\circledast v)b\otimes v'\in \MTA 
\end{equation}
where $\circledast:\left(\mcU/\NR^1\mcU\right)\otimes_{\mcU}\left(\mcU/\NL^1\mcU\right)\to \mathrm{A}_0$ is defined on elements with homogeneous degree by:
\begin{equation}\label{acircle*b}
    a\circledast b:=\left\{\begin{array}{ll}
        0 & \text{if}\ \deg(a)+\deg(b)\neq0, \\
        \lbrack ab\rbrack_0  &  \text{otherwise},
    \end{array}\right.
\end{equation}
and extended by linearity.
Here $\lbrack a\rbrack_0$ denotes the image of $a$ in  $\mathrm{A}_0$. 

It is shown in \cite{DGK} in the context of vertex operator algebras of CFT-type that the map $\circledast$ is in fact an isomorphism, and that this extends to $\mathbb{Z}_{\geq 0}$-graded vertex algebras. Here we simply make note of the fact that if $V$ is a $\mathbb{Z}$-graded vertex algebra, then  $\circledast$ is still an isomorphism and the product $\star$ on $\MTA$ defines an associative algebra structure.  Thus we have

\begin{defn}[Mode transition algebra]\label{def:mode_transition_alg}
Let $V$ be a $\mathbb{Z}$-graded vertex algebra.  Then
   the product $\star:\MTA\times\MTA\to\MTA$ defines an algebra structure on $\MTA$, and $(\MTA, \star)$ is called the \emph{mode transition algebra} of $V$.
   Moreover, for $d \in \mathbb{Z}_{\geq 0}$, the subspace $\MTA_d:=\MTA_{d,-d}$ is closed under multiplication. Hence, it defines a subalgebra of $\MTA$ called the \emph{$d$-th mode transition algebra} of $V$.
\end{defn}

Note that the definition of $\star$ can be generalized to define a left and right $\MTA$-module structure on 
\begin{equation}\label{eq:induction_module}
     \Phi^L(W)=\left(\mcU/\NL^1\mcU\right)\otimes_{\mcU_0} W,\qquad
    \Phi^R(M)=M\otimes_{\mcU_0}\left(\mcU/\NR^1\mcU\right)
\end{equation}
for all left and right $\HeZhu{0}$-modules $W$ and $M$ respectively.
We will use this fact in the next Section. 

\subsection{Mode transition algebras with unities and the Splitting Theorem for Zhu algebras}

One of the strengths of the notion of mode transition algebras is that one may use them to determine the relation between the various levels of Zhu algebras as illustrated for the case of the Heisenberg vertex operator algebra in \cite{DGK}, which is dependent upon whether or not these algebras have unities \cite[Theorem 6.0.1]{DGK}.  In this subsection, we generalize these relations in the setting of $\mathbb{Z}$-graded vertex algebras.

Let $V$ be a $\mathbb{Z}$-graded vertex algebra.  For shorthand, we let $\HeZhu{d} := \HeZhu{d}(V)$, for $d\in \mathbb{Z}_{\geq 0}$, and we  
consider the map
\begin{eqnarray}\label{define-mu}
\mu_d : \MTA_d &\longrightarrow& \HeZhu{d} \\
\bar{\alpha} \otimes u \otimes \bar{\beta} & \mapsto & [\alpha u \beta]_d  \nonumber ,
\end{eqnarray}
where  $\bar{\alpha} = \alpha + \NL^1 \mathscr{U} \in (\mathscr{U}/\NL^1 \mathscr{U})_d$, $u \in \HeZhu{0}$, and $\bar{\beta} = \beta + \NR^1 \mathscr{U} \in (\mathscr{U}/\NR^1 \mathscr{U})_{-d}$,  and $[\alpha u \beta]_d$ denotes the equivalence class of $\alpha u \beta$ in $\HeZhu{d}$.
By \cite[Lemma B.3.1]{DGK} extended to the setting of a $\mathbb{Z}$-graded vertex algebra, this map is well-defined and there is an exact sequence
\begin{equation}\label{exact-seq} \MTA_d \xlongrightarrow{\mu_d} \HeZhu{d} \xlongrightarrow{\pi_d} \HeZhu{d-1} \longrightarrow 0
\end{equation}
where $\pi_d: \HeZhu{d}\to  \HeZhu{d-1}$ is the canonical projection.

We next note that the analog of \cite[Theorem B.3.3]{DGK} holds in the setting of $\mathbb{Z}$-graded vertex algebras, which we present as Lemma \ref{injective-lemma} and the Splitting Theorem \ref{thm:unity} below.   

\begin{lem}\label{injective-lemma}  
Let $V$ be a $\mathbb{Z}$-graded vertex algebra, and assume that for $d \in \mathbb{Z}_{\geq 0}$, the ring $\MTA_d$ has a unity $\mathscr{I}_d$.  Then the map $\mu_d : \MTA_d \longrightarrow \HeZhu{d}$ given by Eqn.\ \eqref{define-mu} is injective, i.e., the exact sequence \eqref{exact-seq} is a short exact sequence
\begin{equation}\label{short-exact} 
0 \longrightarrow \MTA_d \xlongrightarrow{\mu_d} \HeZhu{d} \xlongrightarrow{\pi_d} \HeZhu{d-1} \longrightarrow 0.
\end{equation}
\end{lem}

\begin{proof}
As in \cite{DGK} we have that for any left $\HeZhu{0}$-module  $W$, we can form the $\MTA$-module $\Phi^L(W)$ defined as in Eqn. \eqref{eq:induction_module},
where the action of $\MTA_d$ factors through the action of $\HeZhu{d}$.  
Let $\mathfrak{a} \in \MTA_d$ such that $\mu_d(\mathfrak{a}) = 0$.  Then $\mathfrak{a}$ acts on $\Phi^L(W)_d$ trivially.  
In particular, for $W$ being the $\HeZhu{0}$-module given by
\[
W = \Phi^R(\HeZhu{0}) = \HeZhu{0} \otimes _{\mcU_0}  \mcU/\NR^1\mcU ,
\]
this implies that the action of $\mathfrak{a}$ on $\MTA_d\subset  \left( \mcU/\NL^1\mcU \otimes_{\mcU_0}\HeZhu{0} \otimes_{\mcU_0}\mcU/\NR^1\mcU\right)_d = \Phi^L(W)_d$ is trivial.  
By definition, this action is nothing more than the algebra left product by $\mathfrak{a}$ on $\MTA_d$.
Since $\MTA_d$ has a unity $\mathscr{I}_d$, we have that $\mathfrak{a} \star \mathscr{I}_d = \mathfrak{a}$ is zero in $\MTA_d$. 
Hence, the kernel of $\mu_d$ is trivial. 
\end{proof}

We have the following Theorem, which extends \cite[Proposition 6.0.1(a) or Theorem B.3.3]{DGK} in the setting of vertex operator algebras of CFT-type to the setting of $\mathbb{Z}$-graded vertex algebras. 

\begin{thm}\label{thm:unity} Let $V$ be a $\mathbb{Z}$-graded vertex algebra and let $d \in \mathbb{Z}_{\geq 0}$.  
If the $d$-th mode transition algebra $\MTA_d$ of $V$ admits an identity element (i.e., a unity), then we have an isomorphism of rings
\begin{equation*}
    \HeZhu{d}(V)\cong \MTA_d\times \HeZhu{d-1}(V).
\end{equation*}
In particular, if $\MTA_j$ admits a unity for all $0\leq j\leq d$, then
\begin{equation*}
    \HeZhu{d}(V)\cong \bigoplus_{j=0}^d\MTA_j.
\end{equation*}
\end{thm}

\begin{proof}
Since $\mu_d$ is injective,  $\mu_d(\mathscr{I}_d) \in \HeZhu{d}$ is nonzero and a nontrivial idempotent.  Thus
\begin{equation}
\HeZhu{d}(V) \cong \mu_d(\mathscr{I}_d) \HeZhu{d}(V) \times ({\bf 1}(-1) - \mu_d(\mathscr{I}_d))  \HeZhu{d}(V) \cong \MTA_d \times \HeZhu{d-1}(V).  
\end{equation}
The rest follows.  
\end{proof}

\subsection{Families of strong unities for mode transition algebras}\label{sec:strong_id}

In \cite{DGK}, it was shown that in the setting of vertex operator algebras of CFT-type, if the mode transition algebras admit a \emph{family of strong unities}, then certain properties of geometric constructions of moduli of curves related to ``admissible" $V$-modules can be determined.  Here we extend the notion of a \emph{family of strong unities} introduced in \cite{DGK} to the setting of $\mathbb{Z}$-graded vertex algebras. We also show that certain properties are equivalent to the existence of a family of strong unities analogous to those given in \cite{DGK}, although our equivalent notions contain some slight modifications compared to those in \cite{DGK}. 

\begin{defn}\label{def:strong_id}
Let $V$ be a $\mathbb{Z}$-graded vertex algebra, and assume that for every $d \in \mathbb{Z}_{\geq 0}$, the ring $\mathfrak{A}_d$ associated to $V$ is unital, with unity $\mathscr{I}_d \in \mathfrak{A}_d$.  We say that the family of unities $\{\mathscr{I}_d\}_{d\in  \mathbb{Z}_{\geq 0}}$ is a {\it family of strong unities for}  $\MTA$ if for every $m,n \in \mathbb{Z}_{\geq 0}$ and for every $\mathfrak{a} \in \mathfrak{A}_{n, -m}$,
\begin{equation}\label{strong-id}
\mathscr{I}_n \star \mathfrak{a} = \mathfrak{a} = \mathfrak{a} \star \mathscr{I}_m. 
\end{equation}
\end{defn}

When $V$ is a vertex operator algebra of {CFT type}\footnote{ A vertex operator algebra $V$ is said of \emph{CFT type} is $V=\bigoplus_{n\in\Z_{\geq0}}V_n$ and $V_0=\C\vac$, cf. \cite{DLMM}.}, \cite[Definition/Lemma 3.3.1]{DGK} both defines strong unities and explores other conditions for the unities $\mathscr{I}_d \in \mathfrak{A}_d$ to be strong unities.  Here, we have extended the definition of a family of strong unities to the setting of  $\Z$-graded vertex algebras by first considering what we believe is the most natural condition to be the definition of a family of strong unities. 

Recall that, for a left $\HeZhu{0}$-module $W$ and a right $\HeZhu{0}$-module $M$,  
$\Phi^L(W) = \mathscr{U}/\NL^1\mathscr{U} \otimes_{\mathscr{U}_0} W$ and 
$\Phi^R(M) = M \otimes_{\mathscr{U}_0} \mathscr{U}/\NR^1\mathscr{U}$
defined in Eqn.\eqref{eq:induction_module} are respectively left and right $\MTA$-modules.
We recall the following associativity result, a version of which appeared  in \cite[Proposition B.2.5]{DGK}, and that we apply to a $\mathbb{Z}$-graded vertex algebra.

\begin{lem}   \label{assoc-lemma}
Let $V$ be a $\mathbb{Z}$-graded vertex algebra and let $d, m, n \in \mathbb{Z}_{\geq 0}$. Then for all $u \in \mathfrak{L}(V)^f_d$, $\mathfrak{b} \in \MTA_{n}$, $\mathfrak{c} \in \MTA_{m}$  and $\mathfrak{a} \in \MTA_{n,-m}$, 
\[ 
(u \cdot \mathfrak{b}) \star \mathfrak{a}  = u \cdot (\mathfrak{b} \star \mathfrak{a} ) \quad \mbox{and} \quad \mathfrak{a} \star (\mathfrak{c} \cdot u) = (\mathfrak{a} \star \mathfrak{c}) \cdot u,
\] 
where $u$ is identified with its image in $\mcU$ (and its quotients) and $\cdot $ denotes the left (respectively the right) $\mcU$-module action on $\mcU/\NL^1\mcU$ (respectively $\mcU/\NR^1\mcU$).
\end{lem}
\begin{proof}
Let $\mathfrak{b}=w_L\otimes b\otimes w_R \in \mathfrak{A}_n$  and $\mathfrak{a}=v_L\otimes a\otimes v_R$ with $w_L,v_L\in (\mcU/\NL^1\mcU)_{n}$, $w_R\in (\mcU/\NR^1\mcU)_{-n}$, $v_R\in(\mcU/\NR^1\mcU)_{-m}$ and $b,a\in\HeZhu{0}$. Then, using the definition \eqref{eq:star_product} of the $\star$ product, we have  \begin{eqnarray*}
        (u \cdot \mathfrak{b}) \star \mathfrak{a}&=&\left( (u\ w_L)\otimes b\otimes w_R\right)\star \mathfrak{a}\\
        &=&(u\ w_L)\otimes b(w_R\circledast v_L) a\otimes v_R\\
        &=&u\cdot(w_L\otimes b(w_R\circledast v_L) a\otimes v_R)=u\cdot (\mathfrak{b} \star \mathfrak{a}).
    \end{eqnarray*}
By linearity, it follows that the associativity $(u\cdot \mathfrak{b})\star \mathfrak{a}=u\cdot (\mathfrak{b} \star \mathfrak{a} )$ holds for any $\mathfrak{b}\in \mathfrak{A}_n$. The proof of the other equality is totally analogous.
\end{proof}

Next we prove some properties of a mode transition algebra $\MTA$ which are equivalent to $\MTA$ having a family of strong unities.   A similar result was given in \cite{DGK} Definition/Lemma 3.3.1, but with an assumption that left unities are also right unities.  Here we do not make that assumption and thus our properties $(iii)-(v)$ below have extra conditions in comparison to those given in \cite{DGK}.

\begin{prop}\label{properties-of-Is}
Let $V$ be a $\mathbb{Z}$-graded vertex algebra and assume that for every $d \in \mathbb{Z}_{\geq 0}$, the ring $\mathfrak{A}_d$ associated to $V$ is unital, with unity $\mathscr{I}_d \in \mathfrak{A}_d$.  
Then, the following are equivalent: 
\begin{enumerate}[label=(\roman*),leftmargin=*]
    \item The family of unities $\{\mathscr{I}_d\}_{d\in  \mathbb{Z}_{\geq 0}}$ is a family of strong unities.
    
    \item For every $d, m, n \in \mathbb{Z}_{\geq 0}$, for all $u \in \mathfrak{L}(V)^f_d$ and $\mathfrak{a} \in \MTA_{n,-m}$ 
\[ (u \cdot \mathscr{I}_n) \star \mathfrak{a}  = u \cdot \mathfrak{a} \quad \mbox{and} \quad \mathfrak{a} \star (\mathscr{I}_m \cdot u) = \mathfrak{a} \cdot u.\] 

    \item For every $d \in \mathbb{Z}_{\geq 0}$, the homomorphisms $\mathfrak{A}_d \rightarrow \mathrm{End}(\Phi^L(\HeZhu{0})_d)$ and $\mathfrak{A}_d \rightarrow \mathrm{End}(\Phi^R(\HeZhu{0})_{-d})$ are unital.
    
    \item For every $d \in \mathbb{Z}_{\geq 0}$, the homomorphisms $\mathfrak{A}_d \rightarrow \mathrm{End}(\Phi^L(\HeZhu{0})_d)$ and $\mathfrak{A}_d \rightarrow \mathrm{End}(\Phi^R(\HeZhu{0})_{-d})$ are unital and injective.
    
    \item For every $d \in \mathbb{Z}_{\geq 0}$ and $\HeZhu{0}$-module $W$, the homomorphisms $\mathfrak{A}_d \rightarrow \mathrm{End}(\Phi^L(W)_d)$ and $\mathfrak{A}_d \rightarrow \mathrm{End}(\Phi^R(W)_{-d})$ are unital.
\end{enumerate}
\end{prop}

\begin{proof} 
For the proof, we follow the method of \cite{DGK}, and observe that these properties hold in the setting of $\mathbb{Z}$-graded vertex algebras.  

We first note that condition $(iv)$ implies $(iii)$, and $(v)$ implies $(iii)$. Thus it will be enough to show the following implications:
\[ (ii) \iff (i) \iff (iii) \quad \mbox{and} \quad (iii) \implies (v) \quad \mbox{and} \quad (iii) \implies (iv).\]
\begin{description}
    \item[$(ii) \implies (i)$] follows from taking $d = 0$ and $u = {\bf 1}(-1) \in \mathfrak{L}(V)^f_0$.
    
    \item[$(i) \implies (ii)$] 
    %{\J it is enough to show that $(u \cdot \mathscr{I}_n) \star \mathfrak{a}=u\cdot (\mathscr{I}_n \star \mathfrak{a})$ and $\mathfrak{a} \star (\mathscr{I}_m \cdot u) = (\mathfrak{a} \star \mathscr{I}_m) \cdot u$ as $\mathscr{I}_m$ and $\mathscr{I}_n$ are strong identities. 
    follows from the associativity Lemma \ref{assoc-lemma} taking $\mathfrak{b}=\mathscr{I}_n$, $\mathfrak{c}=\mathscr{I}_m$ and using the fact that $\mathscr{I}_n$ and $\mathscr{I}_m$  are strong unities.  
    
%We shall focus on the first relation as the proof of the second one is similar.
  %  \textcolor{teal}{Assume first} $\mathscr{I}_n=\iota_L\otimes\iota_0\otimes\iota_R$  and $\mathfrak{a}=v_L\otimes a\otimes v_R$ with $\iota_L,v_L\in (\mcU/\NL^1\mcU)_{n}$, $\iota_R\in (\mcU/\NR^1\mcU)_{-n}$, $v_R\in(\mcU/\NR^1\mcU)_{-m}$ and $\iota_0,a\in\HeZhu{0}$. Then, using the definition \eqref{eq:star_product} of the $\star$ product, we have
   % \begin{eqnarray*}
    %    (u \cdot \mathscr{I}_n) \star \mathfrak{a}&=&\left( (u\ \iota_L)\otimes \iota_0\otimes \iota_R\right)\star \mathfrak{a}\\
     %5   &=&(u\ \iota_L)\otimes %\iota_0(\iota_R\circledast v_L) a\otimes v_R\\
     %   &=&u\cdot(\iota_L\otimes \iota_0(\iota_R\circledast v_L) a\otimes v_R)=u\cdot (\mathscr{I}_n \star \mathfrak{a}).
    %\end{eqnarray*}}
    %\textcolor{teal}{In general, $\mathscr{I}_n$ is a linear combination of elements of the form $\iota_L\otimes\iota_0\otimes\iota_R$. We can show that the computation still holds by linearity. We leave this as an exercise for the reader.}

    \item[$(i) \implies (iii)$] note that 
    \begin{eqnarray*}
        \Phi^L(\HeZhu{0})_d &=& \left(\mathscr{U}/\NL^1\mathscr{U} \otimes_{\mathscr{U}_0} \HeZhu{0} \right)_d \ = \  \left(\mathscr{U}/\NL^1\mathscr{U}\right)_d \otimes_{\mathscr{U}_0} \HeZhu{0} \\
        &\cong& \left(\mathscr{U}/\NL^1\mathscr{U}\right)_d \otimes_{\mathscr{U}_0} \HeZhu{0}  \otimes_{\HeZhu{0}} \HeZhu{0} \\
        &\cong& \left(\mathscr{U}/\NL^1\mathscr{U}\right)_d \otimes_{\mathscr{U}_0} \HeZhu{0}  \otimes_{\mathscr{U}_0}  \left(\mathscr{U}/\NR^1 \mathscr{U} \right)_0 = \MTA_{d,0} .
    \end{eqnarray*} 
    Hence, identifying $\Phi^L(\HeZhu{0})_d\cong \MTA_{d,0}$ maps $\mathscr{I}_d\in\MTA_d$ onto the trivial endomorphism of $\MTA_{d,0}$. Similarly, one can identify $\Phi^R(\HeZhu{0})_{-d}\cong \MTA_{0,-d}$, showing that the second homomorphism is also unital.

    \item[$(iii) \implies (i)$]
    using the second equality of Lemma \ref{assoc-lemma} we show that, if $\mathfrak{a}=v_L\otimes a\otimes v_R$ with $v_L\in (\mcU/\NL^1\mcU)_{n}$, $v_R\in(\mcU/\NR^1\mcU)_{-m}$ and $a\in\HeZhu{0}$, then
    \[\mathscr{I}_n \star (v_L \otimes a \otimes v_R) = (\mathscr{I}_n \star (v_L \otimes a \otimes 1)) \cdot v_R.\]
    %let $\mathscr{I}_n=\iota_L\otimes\iota_0\otimes\iota_R$ and $\mathfrak{a}=v_L\otimes a\otimes v_R$ with $\iota_L,v_L\in (\mcU/\NL^1\mcU)_{n}$, $\iota_R\in (\mcU/\NR^1\mcU)_{-n}$, $v_R\in(\mcU/\NR^1\mcU)_{-m}$ and $\iota_0,a\in\HeZhu{0}$.
    %Then, using \eqref{eq:star_product} again, we have
    %\begin{eqnarray*}
    %\mathscr{I}_n \star (v_L \otimes a \otimes v_R) 
    %&=& \iota_L\otimes \iota_0(\iota_R\circledast v_L) a\otimes (1\ v_R)\\
    %&=& (\iota_L\otimes \iota_0(\iota_R\circledast v_L) a\otimes 1)\cdot v_R\\
    %&=& (\mathscr{I}_n \star (v_L \otimes a \otimes 1)) \cdot v_R.
    %\end{eqnarray*}
    Since $\mathscr{I}_n$ acts trivially on $v_L \otimes a \otimes 1\in\MTA_{n,0}\cong \Phi^L(\HeZhu{0})_n$ by the previous identification, we deduce
    \[
    \mathscr{I}_n \star \mathfrak{a}
    = (v_L \otimes a \otimes 1) \cdot v_R
    = \mathfrak{a},
    \]
    which holds for any $\mathfrak{a}\in\MTA_{n,-m}$ by linearity.
    We prove the second identity relation $\mathfrak{a}\star\mathscr{I}_m=\mathfrak{a}$ using a similar argument, with $\mathscr{I}_m$ acting on a element of $\MTA_{0,-m}\cong\Phi^R(\HeZhu{0})_{-m}$.

    \item[$(iii) \implies (v)$] let $u \otimes w \in \Phi^L(W)_d$, with $u\in (\mcU/\NL^1\mcU)_d$ and note that $u\otimes w=(u\otimes 1_{\HeZhu{0}})\cdot(1_{\mcU}\otimes w)$ where $\cdot$ is the natural action of $\Phi^L(\HeZhu{0})_d$ onto $\Phi^L(W)_d$.
    The action of the mode transition algebra $\MTA_d$ acts by left multiplication on $\Phi^L(W)_d$ is compatible with the action $\cdot$ in the following sense:
    \begin{eqnarray*}
        \mathscr{I}_d \star (u \otimes w)
        &=& \iota_L\otimes \iota_0(\iota_R\circledast u) w \\
        &=& (\iota_L\otimes \iota_0(\iota_R\circledast u) 1_{\HeZhu{0}})\cdot (1_\mcU\otimes w)\\
        &=& \left(\mathscr{I}_d \star (u \otimes 1_{\HeZhu{0}})\right)\cdot (1_\mcU\otimes w).
    \end{eqnarray*}  
    Since $\mathscr{I}_d$ acts trivially on $\Phi^L(\HeZhu{0})_d$, the first factor reduces to
    $\mathscr{I}_d \star (u \otimes 1_{\HeZhu{0}})=(u \otimes 1_{\HeZhu{0}})$ and thus the homomorphism $\mathfrak{A}_d \rightarrow \mathrm{End}(\Phi^L(W)_d)$ is unital. 
    A similar argument using the action of $\Phi^R(\HeZhu{0})_{-d}$ by right multiplication onto $\Phi^R(M)_{-d}$ allows us to prove that the second homomorphism is also unital.

    \item[$(iii) \implies (iv)$]  we need to show that the two homomorphisms are injective.
    The action of the left multiplication by $\star$ of $\MTA_d$ on $\MTA\simeq \Phi^L(\HeZhu{0})\otimes \mcU/\NR^1\mcU$ is determined by the left multiplication by $\star$ on $\Phi^L(\HeZhu{0})$ only.
    Additionally, we showed in the proof of Lemma~\ref{injective-lemma} that when $\MTA_d$ has a unit, $\mathfrak{a}\in\MTA_d$ acts trivially if and only if $\mathfrak{a}=0$ has
    $$\mathfrak{a}=\mathfrak{a}\star\mathscr{I}_d=0.$$
    So the former action is injective as well as the latter.
    Again, one can adapt the argument to prove that the second homomorphism is injective as well.
\end{description}
\end{proof}

In the next Section, we construct the $d$-th mode transition algebras for a $\mathbb{Z}$-graded vertex algebra that is not a vertex operator algebra, but is a $\mathbb{Z}_{\geq 0}$-graded conformal vertex algebra with infinite dimensional weight spaces, namely for the Weyl vertex algebra at central charge 2. In addition, we prove that these algebras are unitary and that the unities form a family of strong unities.  

\section{The mode transition algebra and higher level Zhu algebras for the Weyl vertex algebra at central charge \texorpdfstring{$2$}{2}.}\label{sec:WeylVA}

In this Section, we introduce the Weyl vertex algebra $\WeylVA$ also known as the $\beta\gamma$-system or bosonic ghosts.
This vertex algebra has a one-parameter family of conformal vectors but we shall focus here on the one with conformal charge $c=2$.
With this choice of conformal element, the Weyl vertex algebra $\WeylVA$ is a conformal $\mathbb{Z}$-graded vertex algebra, and in fact a conformal $\mathbb{Z}_{\geq 0}$-graded vertex algebra, with infinite dimensional graded spaces.

We give the structure of the mode transition algebras at all levels for the Weyl vertex algebra $\WeylVA$.  We prove that these algebras are unital and that the unities provide a family of strong unities (see Section \ref{sec:strong_id}).
Using Theorem \ref{thm:unity}, we describe all the Zhu algebras for $\WeylVA$. 

A similar result, conjectured by Addabbo and Barron \cite{AB23}, for the Heisenberg vertex operator algebra was proved in \cite{DGK} using mode transition algebras.
It has applications to algebro-geometric settings. 
The Weyl vertex algebra $\WeylVA$ at central charge 2 is not $C_1$-cofinite (since $\WeylVA_0$ is infinite dimensional), even though its mode transition algebra admits a family of strong unities as we show below. It is not clear what algebro-geometric implications follow.

\subsection{The Weyl vertex algebra}\label{subsectionWeyl} 
The \emph{Weyl Lie algebra}, denoted $\Weylalg_{Lie}$, is the 3-dimensional Lie algebra with generators $\{ a, a^*, 1\}$ with the only non-trivial relation given by 
\begin{equation}\label{Weyl-relations} 
[a, a^*] = 1.
\end{equation}
The \emph{rank one Weyl algebra}, denoted $\Weylalg$, is the universal enveloping algebra $\mathcal{U}(\Weylalg_{Lie})$.
Concretely, it corresponds to the quotient of the free unital associative algebra on two generators $a$ and $a^*$ modulo the relation $aa^* - a^*a - 1$:
$$\Weylalg \cong \mathbb{C}\langle a, a^*\rangle/\langle aa^* - a^*a - 1\rangle.$$
Note that $\Weylalg$ is isomorphic to the associative algebra of endomorphisms $\mathbb{C}[x, x^{-1}]$ generated by multiplication by $x$ and the differentiation operator  $\frac{\partial }{\partial x}$ under the correspondence
\begin{equation*}
    \begin{gathered}
        \Weylalg \overset{\sim}{\longrightarrow}\mathbb{C}[x, x^{-1}],\qquad
        a\mapsto\frac{\partial }{\partial x},\qquad a^*\mapsto x.
    \end{gathered}
\end{equation*}

The group of automorphisms of $\Weylalg$ contains infinitely many elements and, 
as proved by Dixmier \cite{D}. 
Among them, the following family of automorphisms of $\Weylalg$ will be useful.
Consider for any $\lambda \in \mathbb{C}^\times$, consider the automorphism 
\begin{equation}\label{eq:varphi_def}
 \varphi_\lambda : \Weylalg \longrightarrow \Weylalg , \qquad a \mapsto \lambda a^*, \qquad a^* \mapsto -\lambda^{-1} a.
\end{equation}
that will be of particular interest. Each of these automorphisms $\varphi_\lambda$ is of order four with $\varphi_\lambda^2 = - \id_{\Weylalg}$ and $\varphi_\lambda^3 = \varphi_{\lambda^{-1}}$. 

The \emph{bosonic ghost Lie algebra} $\Lieghost$ is the infinite-dimensional Lie algebra
\begin{equation*}
    \Lieghost= \Weylalg_{Lie}  \otimes \C[s, s^{-1}]\oplus \C K,
\end{equation*}
generated by a central element $K$ and the elements $a_m:=a \otimes s^m$, and $a^*_n:=a^* \otimes s^n$ with $m,n \in \mathbb{Z}$ satisfying
\[ \left[ a_m, a^*_n \right] = \delta_{m+n,0} K.\]
The automorphisms of  $\Weylalg_{Lie}$ can be lifted into $\Lieghost$-automorphisms.
In addition, $\Lieghost$ admits the spectral flow automorphisms 
\begin{equation}\label{spectral-flow}
\sigma^\ell: \Lieghost \longrightarrow \Lieghost \qquad a_n \mapsto a_{n+ \ell} \qquad a^*_n \mapsto a^*_{n-\ell}, \qquad K\mapsto K,    
\end{equation} 
for $\ell \in \mathbb{Z}$. 

Following \cite{RW}, we fix the triangular decomposition
\begin{equation*}
    \Lieghost=\Lieghost^-\oplus\Lieghost^0\oplus\Lieghost^+ 
\end{equation*}
where 
\begin{align*}
 \Lieghost^{-}&= \mathrm{span}_{\mathbb{C}}\{ a_{-m}, a^*_{-n} \, | \, m\geq1 ,  n \geq 0 \}, \\
 \Lieghost^{0}&= \mathrm{span}_{\mathbb{C}}\{  K \}, \\
 \Lieghost^{+}&= \mathrm{span}_{\mathbb{C}}\{ a_{m}, a^*_{n} \, | \, m \geq 0, n\geq 1 \}.
\end{align*}
The \emph{rank one Weyl vertex algebra}, denoted $\WeylVA$, is the induced $\Lieghost$-module 
$$\WeylVA = \mathcal{U}(\mathcal{L}) \otimes_{\mathcal{U}(\Lieghost^0\oplus \Lieghost^+ )} \mathbb C \bf 1,$$
where $\mathbb{C} \bf 1$ is the one-dimensional $ \Lieghost^0\oplus \Lieghost^+$-module on which $\Lieghost^{+}$ acts trivially and $K$ acts as the multiplication by the scalar $1$:
\begin{align*}
a_m{\bf 1}&=0 \text{ for }m \geq 0, \\
a^*_n{\bf 1}&=0  \text{ for }n \geq 1,\\
K\vac&=\vac.
\end{align*}
Then, $\WeylVA\cong U(\mathcal{L}^-)$ as vector spaces, and $\WeylVA$ admits a PBW basis which consists of the monomials
\begin{equation}\label{eq:PBW}
    a_{-m_1-1}a_{-m_2-1} \dots a_{-m_k-1}a^*_{-n_1}\dots a^*_{-n_l}\vac,
\end{equation}
where $m_i,n_j\in\Z$ with
$m_1\geq m_2\geq\dots\geq m_k\geq0$,  $n_1\geq n_2\geq\dots\geq n_l\geq0$, and, abusing notation, $\vac$  denotes the image of $1\otimes \vac$ in $\WeylVA$.
There is a unique vertex algebra structure on  $\WeylVA$ (see for instance Theorem 5.7.1 in \cite{LL} or Lemma 11.3.8 in \cite{FBbook}) given by $(\WeylVA,Y,{\bf 1})$ with vertex operator map 
\[Y:  \WeylVA\longrightarrow \mathrm{End}(\WeylVA)[[z,z^{-1}]]\] 
such that 
\begin{equation}
\begin{aligned}\label{genweyl}
Y(a_{-1}{\bf 1},z)=a(z):=\sum_{n\in\mathbb{Z}}a_nz^{-n-1}, \qquad & Y(a^*_0{\bf 1},z)=a^*(z):=\sum_{n\in\mathbb{Z}}a^*_nz^{-n}.    
\end{aligned}
\end{equation}
The fields $a(z)$, $a^*(z)$ satisfy the operator product expansions
\begin{equation}
    \begin{gathered}
        a(z) a^*(w) = \frac{1}{z-w} +  \no{a(z) a^*(w)} \\
        a(z) a(w) =\no{a(z) a(w)},\qquad a^*(z) a^*(w) =\no{a^*(z) a^*(w)}.
    \end{gathered}
\end{equation}
where $\no{x(z)y(z)}$ denotes the ordered product of the fields $x(z)$ and $y(z)$ given by
\begin{align*}
\no{x(z)y(z)}\overset{\text{def}}{=} \ x(z)^+y(z)+y(z)x(z)^-,
\end{align*} 
where $x(z)^+=\sum_{n\leq -1}x_nz^{-n-1}$ and $x(z)^-=\sum_{n\geq 0}x_nz^{-n-1}$.
Moreover, the map $Y:  \WeylVA \rightarrow \mathrm{End}(\WeylVA)[[z,z^{-1}]]$ is defined by linearity on the elements of the basis \eqref{eq:PBW} by
\begin{align*}
Y(a_{-m_1-1}a_{-m_2-1}& \dots a_{-m_k-1}a^*_{-n_1}\dots a^*_{-n_l}\vac, z)\\
& = \prod_{i=1}^k \frac{1}{m_i!} \prod_{j=1}^\ell \frac{1}{n_j!} \no{\partial^{m_1}a(z)\dots  \partial^{m_k}a(z) \partial^{n_1}a^*(z)\dots\partial^{n_l}a^*(z)}.  
\end{align*}

\begin{rem}\label{bgrem}
\phantom{x}
\begin{enumerate}[label=(\roman*),leftmargin=*]
    \item The fields $a(z)$ and $a^*(z)$ defined in \eqref{genweyl} are usually denoted by $\beta(z)$ and $\gamma(z)$ in the Physics literature (up to a choice of sign) where the vertex algebra $\WeylVA$ is referred to as the $\beta\gamma$ vertex algebra, $\beta \gamma$-system, or bosonic ghosts.
    
    \item Since for all $n\in \mathbb{Z}$, the $n$ modes of the fields $Y(a_{-1}\vac, z)=a(z)$, $Y(a^*_0\vac,z)= a^*(z)$ satisfy 
\begin{align*}
    (a_{-1}\vac)_n=a_n, \qquad (a^*_0\vac)_n=a^*_{n+1},
\end{align*}
we have that $G=\{a_{-1}\vac, a^*_0\vac\}$ is a set of strong generators for the vertex algebra $\WeylVA$.
Namely, $\WeylVA$ is spanned by the set of normally ordered monomials 
$$\{ :(\partial^{k_1} \alpha^{i_1}) \dots (\partial^{k_l} \alpha^{i_l}): \mid \  k_1, \dots, k_l \geq 0, \  \alpha^{i_j}\in  G \}.$$ 
 Therefore, $\WeylVA$ is strongly finitely generated as a vertex algebra.
\end{enumerate}
\end{rem}

The {\it affinized rank one Weyl algebra}  $ \widehat{\Weylalg}$ is the algebra
\begin{equation}\label{A-hat}
\widehat{\Weylalg} = \frac{\mathcal{U}(\Lieghost)}{\langle K-1\rangle },
\end{equation}
where  $\mathcal{U}(\Lieghost)$ denotes the  universal enveloping algebra of $\Lieghost$ and $\langle K - 1\rangle$ is the two sided ideal generated by $K - 1$.
The algebra $\widehat{\Weylalg}$ corresponds to the mode algebra of $\WeylVA$. 
It is an associative algebra with generators $a_m$, $a^*_n$,  for $m, n\in\mathbb{Z}$, and relations
\begin{equation} \label{com}
[a_m,a^*_n]= \delta_{m+n,0}, \qquad
 \left[a_m,a_n \right]= [a^*_m,a^*_n] \ = \ 0.
\end{equation}
The spectral flow automorphisms \eqref{spectral-flow} of the Lie algebra $\mathcal{L}$ lift to automorphisms of $\mathcal{U}(\mathcal{L})$ (and thus of $\widehat{\Weylalg}$) via the Poincaré-Birkhoff-Witt theorem.

From the simple relations between the modes of the strong generators $a_{-1}\vac$ and $a^*_0\vac$ given by \eqref{com} together with Remark \ref{bgrem}, it follows that $\WeylVA$ is a simple vertex algebra.  

Set $J=a_{-1}a^*_0\vac$.
The corresponding field $J(z)=Y(J,z)=\sum_{n\in\mathbb{Z}}J_nz^{-n-1}$ is a Heisenberg vector in $\WeylVA$ of level $-1$ (i.e., $J$ generates a Heisenberg Lie algebra of central charge $-1$).
For $n,m\in\mathbb{Z}$, we have 
\begin{equation}\label{eq:heisenberg_field}
    [J_m,J_n]=-m\delta_{m+n,0}
\end{equation}
as operators on $\WeylVA$, and therefore
\begin{equation}
    J(z)J(w)=-\frac{1}{(z-w)^2}+\no{J(z)J(w)}.
\end{equation}
In addition, we have for all $m,n\in\Z$,
\begin{equation}\label{eq:J_grading}
    [J_m,a_n]=-a_{m+n}, \quad \mbox{ and} \quad [J_m,a^*_n]=a^*_{m+n}.
\end{equation}
The vertex algebra $\WeylVA$ admits a one parameter family of Virasoro vectors
\begin{equation} \label{conformal}
\omega_{\mu} 
=  a_{-1}a^*_{-1}\vac-\mu J_{-2}\vac,\qquad(\mu\in\C)
\end{equation}
with central charge 
\begin{equation} \label{central}
c_{\mu}=2(6\mu(\mu-1)+1).   
\end{equation}
The conformal vertex algebra $\WeylVA$ corresponding to the choice of the conformal vector given by $\omega_{\mu}$ is denoted by $_\mu \WeylVA$ following \cite{BBOHPTY}.

\begin{rem}
In \cite{BBOHPTY}, the affinized rank one Weyl algebra $\widehat{\Weylalg}$ is denoted by $\mathcal{A}_1$ in Section III.  Whereas in Section VI of \cite{BBOHPTY}, there is a typo and the $\mathcal{A}_1$ there should be $\Weylalg$. 
In addition, in \cite{BBOHPTY} the Weyl vertex algebra $_\mu \WeylVA$ was denoted $_\mu M$. 
\end{rem}

In \cite{BBOHPTY} three of the authors of this paper together with Batistelli and Pedi\'c Tomi\'c classified the different $\mathbb{C}$-graded conformal vertex algebra structures obtained on the Weyl vertex algebra $_\mu \WeylVA$ depending on the parameter $\mu$.
As shown in \cite{BBOHPTY}, only when $\mu$ is in a certain convex bounded region of the complex plane does $_\mu \WeylVA$ have the structure of a so-called \emph{admissible $\Omega$-generated $\mathbb{C}$-graded vertex operator algebra}, an analogous notion to that of a vertex operator algebra in the $\mathbb{C}$-graded setting. 
Moreover, for $\mu$ in this bounded convex region, $_\mu \WeylVA$ is rational and admits a unique irreducible ordinary module \cite{BBOHPTY}. 
On the other hand, outside of this specific region, $_\mu \WeylVA$ is only a finitely strongly generated $\mathbb{C}$-graded conformal vertex algebra and not a $\C$-graded analogue of a vertex operator algebra.

In the following, we will focus on the choice of parameter $\mu=0$ so that $_0 \WeylVA$
belongs to the latter family of Weyl vertex algebras.
In particular, it admits infinite dimensional conformal weight spaces. 
Thus, it has the structure of a $\mathbb{Z}_{\geq0}$-graded vertex algebra but not of a vertex {\it operator} algebra.
For simplicity, in the rest of the present paper, we denote $_0\WeylVA$ by $\WeylVA$.
Then, we fix $\omega = \omega_0 = a_{-1} a^*_{-1} {\bf 1}$ so that $\WeylVA$ has central charge $c = c_0 = 2$.
The corresponding Virasoro field $L(z)=Y(\omega,z)=\sum_{n\in\mathbb{Z}}L_nz^{-n-2}$ is given by
\begin{align*}
L(z)=\no{a(z)\partial a^*(z)} 
\end{align*}
and satisfies
\begin{align*}
    L(z) L(w)= \frac{1}{(z-w)^4}+\frac{2L(w)}{(z-w)^2}+\frac{\partial_w L(w)}{z-w}+\no{L(z) L(w)}.
\end{align*}
The vertex algebra $\WeylVA$ is $\Z$-graded by $L_0$ as, for all $m,n\in\Z$,
\begin{equation*}
\left[L_m,a_n \right] = -na_{m+n},\qquad
\left[L_m,a^*_n \right] = -(m+n)a^*_{m+n}.
\end{equation*}
In particular, $L_0$ acts semisimply with
\begin{equation*}
\left[L_0,a_n\right] \ = \ -na_n,\qquad \mbox{and} \qquad 
\left[L_0,a^*_n\right] \ = \ -na^*_n.
\end{equation*}
Notice that for integers $m_1\geq \dots \geq m_k\geq 1$, $n_1\geq \dots \geq n_\ell\geq 0$,  and $k,\ell \in \mathbb{Z}_{>0}$, we have 
\begin{equation*}
\begin{aligned}
    L_0a_{-m_1}\dots a_{-m_k}{\bf 1}&=
    (m_1+ \dots +m_k)a_{-m_1} \dots a_{-m_k}{\bf 1},\\
    L_0a^*_{-n_1}\dots a^*_{-n_\ell}{\bf 1}&=
    (n_1+\dots +n_\ell)a^*_{-n_1}\dots a^*_{-n_\ell}{\bf 1},    
\end{aligned}
\end{equation*}
so that 
\begin{equation}\label{L(0)-grading}
{L_0a_{-m_1} \dots a_{-m_k}a^*_{-n_1} \dots a^*_{-n_\ell}{\bf 1}}=\Big( \sum_{j=1}^km_j+\sum_{j=1}^\ell n_j \Big)  a_{-m_1} \dots a_{-m_k}a^*_{-n_1} \dots a^*_{-n_\ell}{\bf 1}.
\end{equation}
The weight zero subspace of $\WeylVA$ with respect to the $L_0$-grading is 
\begin{equation*}
\WeylVA_0= \mathrm{span}_\mathbb{C} \{ {(a^*_0)}^r {\bf 1} \; | \; r \in \mathbb{Z}_{\geq 0} \}.
\end{equation*}
Since this space $\WeylVA_0$ is infinite dimensional, we have that 
$(\WeylVA_0+C_1(\WeylVA))/C_1(\WeylVA) \subset \WeylVA/C_1(\WeylVA)$ is infinite dimensional where 
\[C_1(\WeylVA) = \mathrm{span} \{ u_{-1}v \in \WeylVA \; | \; u,v\in\WeylVA, \wt(u)>0  \}. \]
Thus in the terminology of the theory of vertex algebras (see, for instance, \cite{Miy2}), $\WeylVA$ is not $C_1$-cofinite. 
Also, the weight $d$ subspace of $\WeylVA$ is given by 
\begin{equation}\label{Md}
    \WeylVA_d =\WeylVA^{\tres}_d\ltimes \WeylVA_0,
\end{equation}
where
\begin{multline*}
\WeylVA^{\tres}_d=
\mathrm{span}_\mathbb{C}\{ a_{-m_1} \dots a_{-m_k}a^*_{-n_1} \dots a^*_{-n_\ell}{\bf 1} \; | \; \sum_{j = 1}^k m_j + \sum_{j=1}^\ell n_j = d, \\
k,\ell \in \mathbb{Z}_{\geq 0},\ m_j, n_j \in \mathbb{Z}_{>0},\ m_1\geq \dots \geq m_k,\ n_1\geq \dots \geq n_\ell \}.
\end{multline*}
The subspace $\WeylVA^{\tres}_d$ is finite dimensional with $\dim \WeylVA_d^{\tres} = |P_2(d)|$, where $P_2(d)$ denotes the set of 2-component multipartitions of $d$.
That is, the set $P_2(d)$ consists in pairs $({\bf \Lambda}^1,{\bf \Lambda}^2)$ of partitions, ${\bf \Lambda}^i=\{\lambda^i_1\geq\lambda^i_2\geq\dots\}$, $\lambda^i_j\in\Z_{\geq0}$, such that $|{\bf \Lambda}^1|+|{\bf \Lambda}^2|=d$. By convention, zero has one  2-component multipartitions denoted by $\emptyset$, while one has two $\{\emptyset,1\}$ and $\{1,\emptyset\}$.
For instance $(4+1,2+2+1)\in P_2(10)$.
This dimension $|P_2(d)|$ of $\WeylVA^{res}_d$ will appear in the $d$th level Zhu algebra of $\WeylVA$.

Furthermore, in addition to the $L_0$-grading,  we have a grading with respect to $J_0$ defined by Eqn.~\eqref{eq:J_grading}. 
Indeed, for all $m_1\geq\dots \geq m_k\geq 1$, $n_1\geq\dots\geq n_\ell \geq 1$, and $k,\ell \in\mathbb{Z}_{>0}$, $r\in\Z_{\geq0}$, we have
\begin{equation*}
J_0 (a_{-m_1} \dots a_{-m_k}a^*_{-n_1} \dots a^*_{-n_\ell} {(a^*_0)}^r {\bf 1})
=(-k+\ell+ r)a_{-m_1} \dots a_{-m_k}a^*_{-n_1}\dots a^*_{-n_\ell} {(a^*_0)}^r {\bf 1}.
\end{equation*} 
We can therefore, further decompose $\WeylVA_d$ into its $J_0$-graded components
\begin{equation}\label{L J grading for M}
\WeylVA=\bigoplus_{d=0}^{\infty}\WeylVA_d=\bigoplus_{d=0}^{\infty}\bigoplus_{j\in\Z}\WeylVA_d^j\end{equation} 
where $\WeylVA_d^j=\{u\in \WeylVA~|~L_0u=du,~J_0u=ju\}$. 
Note that $\dim \WeylVA_d^j<\infty$ and $\WeylVA_d^j=0$ when $j<-d$.
By Eqn. \eqref{L J grading for M}, $\dim{q_J,q_L}(\WeylVA):=\tr_\WeylVA q_J^{J_0}q_L^{L_0-c/24}$ is well defined.
It follows from the isomorphism of vector spaces
\begin{align*}
    \WeylVA\cong \C\lbrack a_{-n-1},a^*_{-n}\mid n\geq0\rbrack {\bf 1}
    \cong \bigotimes_{n\geq0}\C\lbrack a_{-n-1}\rbrack {\bf 1}\otimes\bigotimes_{n\geq0}\C\lbrack a^*_{-n}\rbrack  {\bf 1}
\end{align*}  
that 
\begin{equation*}
\begin{aligned}
    \dim{q_J,q_L}(\WeylVA)
    &=q_L^{-1/12}\left(\prod_{n=0}^\infty\dim_{q_J,q_L}(\C\lbrack a_{-n-1}\rbrack  {\bf 1} )\right)\left(\prod_{n=0}^\infty\dim_{q_J,q_L}(\C\lbrack a^*_{-n}\rbrack  {\bf 1} )\right)\\
    &=q_L^{-1/12} \prod_{n = 1}^\infty (1-q_Jq_L^n)^{-1} (1 - q_J^{-1} q_L^{n-1})^{-1}.
\end{aligned}
\end{equation*}

\begin{rem} In \cite{RW}, it is shown that this doubly-graded trace for $\WeylVA$ and similar doubly-graded traces for certain classes of $\WeylVA$-modules have interesting modular invariance properties under the action of the modular group $SL_2(\mathbb{Z})$.  One of the motivations of the current paper is to see if appropriate classes of $\WeylVA$-modules admit doubly-graded pseudo-traces using, for instance this double grading with respect to $L_0$ and $J_0$, to define the notion analogous to doubly-graded traces in the setting of graded pseudo-traces as defined in, for instance  \cite{BBOHY}.  
In Section \ref{Weyl-modules-first-section}, we show that two recently studied classes of $\WeylVA$-modules do not admit such well-defined nontrivial doubly-graded pseudo-traces since they do not satisfy a key property, namely the indecomposable reducible modules are not weakly interlocked, see Theorem \ref{first-non-interlocked-thm} and Corollary \ref{non-interlocked-cor}.
\end{rem}

\subsection{Mode transition algebras and higher level Zhu algebras for the Weyl vertex algebra} 
In this section, we describe the $d$-th mode transition algebra for the Weyl vertex algebra at central charge 2, $\WeylVA$, for all $d\in \mathbb{Z}_{\geq 0}$.
Since  $\MTA_0(\WeylVA) = \HeZhu{0}(\WeylVA)$ is isomorphic to $\hZhu{0}(\WeylVA)$, we have that $\mathfrak{A}_0 (\WeylVA) \cong \Weylalg$ as shown in \cite{L1} (see also \cite{BBOHPTY}).
To describe $\MTA_d(\WeylVA)$ for $d\geq 1$, we show the following result.

\begin{prop}\label{mode-transition}
For $d \in \mathbb{Z}_{\geq 0}$, there is a natural algebra isomorphism
\begin{equation} \label{mat}
\mathfrak{A}_d(\WeylVA) \cong \Mat_{|P_2(d)|}(\Weylalg) \cong \Weylalg \otimes_{\mathbb{C}} \Mat_{|P_2(d)|}(\mathbb{C}),
\end{equation}
where $|P_2(d)|$ is the number of 2-component multipartitions of $d$ as in Section \ref{subsectionWeyl}. 
In particular,  the $d$-th mode transition algebra, $\mathfrak{A}_d(\WeylVA)$, is unital for every $d \in \mathbb{Z}_{\geq 0}$.
\end{prop}

\begin{proof}
 
Elements of $\MTA_d(\WeylVA)$ are sums of elements of the form  
\begin{multline}
\epsilon^w_{(({\bf m}_k, {\bf n}_\ell), ({\bf q}_s, {\bf r}_t))} = a(-m_1)  \dots  a(-m_k) a^*(-n_1)  \dots  a^*(-n_\ell) \otimes w\\  
\otimes a^*(q_s) \dots  a^*(q_1)  a(r_t)  \dots a(r_1) 
\end{multline}
where $({\bf m}_k, {\bf n}_\ell) = \left( (m_1, \dots, m_k), (n_1, \dots, n_\ell)\right)$ and $({\bf q}_s, {\bf r}_t)=\left( (q_1, \dots, q_s), (r_1, \dots, r_t)\right)$ are two elements in $ P_2(d)$, $w\in\Weylalg$, and where $a(m)$ denotes the equivalence class of the element $a_{-1}{\bf 1} \otimes t^m$ in $\mathfrak{L}(V)^f$ of degree $-m$, and $a^*(n)$ denotes the equivalence class of the element $a^*_{-1}{\bf 1} \otimes t^n$ in $\mathfrak{L}(V)^f$ of degree $-n$. 
We have that 
\begin{equation}\label{Ud-Weyl-set} 
\MTA_d(\WeylVA) = \Span_{\C}\{ 
\epsilon^w_{(({\bf m}_k, {\bf n}_\ell), ({\bf q}_s, {\bf r}_t))} \mid ({\bf m}_k, {\bf n}_\ell), ({\bf q}_s, {\bf r}_t)\in P_2(d), \ w \in \Weylalg \}.
\end{equation}
Moreover, we can define a $\Weylalg$-module structure on $\mathfrak{A}_d(\WeylVA)$ by 
\begin{equation*}
w'.\epsilon^w_{(({\bf m}_k, {\bf n}_\ell), ({\bf q}_s, {\bf r}_t))} = \epsilon^{w'.w}_{(({\bf m}_k, {\bf n}_\ell), ({\bf q}_s, {\bf r}_t))}
\end{equation*}
for $({\bf m}_k, {\bf n}_\ell), ({\bf q}_s, {\bf r}_t)\in P_2(d)$ and  $w,w' \in \Weylalg$, so that 
the set 
\begin{equation}\label{epsilons} 
\{\epsilon_{(({\bf m}_k, {\bf n}_\ell), ({\bf q}_s, {\bf r}_t))}:=\epsilon^1_{(({\bf m}_k, {\bf n}_\ell), ({\bf q}_s, {\bf r}_t))} \}_{({\bf m}_k, {\bf n}_\ell), ({\bf q}_s, {\bf r}_t) \in P_2(d)}
\end{equation}
freely generates $\MTA_d(\WeylVA)$ as a $\Weylalg$-module, i.e., as $\Weylalg$-module 
\begin{equation*}
\mathfrak{A}_d(\WeylVA) \cong \Weylalg \otimes_{\mathbb{C}} \mathrm{span}_\mathbb{C} \{\epsilon_{(({\bf m}_k, {\bf n}_\ell), ({\bf q}_s, {\bf r}_t))} \}_{({\bf m}_k, {\bf n}_\ell), ({\bf q}_s, {\bf r}_t) \in P_2(d)}.
\end{equation*} 

The algebra structure of $\MTA_d(\WeylVA)$ is given by the product $\star$, defined in Section \ref{sec:MTA_Zgraded}, which depends on the product $\circledast$ of two elements in $\mathscr{U}(\WeylVA)$ , resulting in an element in $A_0(\WeylVA)= \Weylalg$.
In our setting, we have
\begin{multline*} 
a^*(q_s)  \dots  a^*(q_1)  a(r_t)  \dots a(r_1) \circledast a(-m_1)  \dots  a(-m_k) 
a^*(-n_1)  \dots  a^*(-n_\ell) \\
    =\begin{cases}
    c({\bf m}_k, {\bf n}_\ell) & \text{if } ({\bf q}_s, {\bf r}_t)= ({\bf m}_k, {\bf n}_\ell),\\
    0 & \text{otherwise},
    \end{cases}
\end{multline*} 
where $c({\bf m}_k, {\bf n}_\ell)$ is a positive integer that only depends on $({\bf m}_k, {\bf n}_\ell)$. 
Thus in $\mathfrak{A}_d(\WeylVA)$,
\begin{equation}\label{circle-product}
 \epsilon_{(({\bf m}'_{k'}, {\bf n}'_{\ell'}), ({\bf q}_s, {\bf r}_t))} \star \epsilon_{(({\bf m}_k, {\bf n}_\ell), ({\bf q}'_{s'}, {\bf r}'_{t'}))} 
    =\begin{cases}
        c({\bf m}_k, {\bf n}_\ell) \epsilon_{(({\bf m}'_{k'}, {\bf n}'_{\ell'}), ({\bf q}'_{s'}, {\bf r}'_{t'}))} & \text{if } ({\bf q}_s, {\bf r}_t)= ({\bf m}_k, {\bf n}_\ell),\\
        0 & \text{otherwise}.
    \end{cases}
\end{equation}
Therefore, there is an algebra isomorphism between the space spanned by the vectors \eqref{epsilons} and $\Mat_{|P_2(d)|}(\mathbb{C})$, identifying $\epsilon_{(({\bf m}_k, {\bf n}_\ell), ({\bf q}_s, {\bf r}_t))}$ with the elementary matrix denoted $\mathfrak{I}_{(({\bf m}_k, {\bf n}_\ell), ({\bf q}_s, {\bf r}_t))}$ in $\Mat_{|P_2(d)|}(\mathbb{C})$ having $ \sqrt{c({\bf m}_k, {\bf n}_\ell)}\sqrt{c ({\bf q}_s, {\bf r}_t))} $ in the $(({\bf m}_k, {\bf n}_\ell), ({\bf q}_s, {\bf r}_t))$-entry, and zero otherwise. 
Thus as an algebra $\mathfrak{A}_d(\WeylVA) \cong \Weylalg \otimes \Mat_{|P_2(d)|}(\mathbb{C})$.  

In addition, this implies that  $\mathfrak{A}_d(\WeylVA)$ is unital for every $d \in \mathbb{Z}_{\geq 0}$  with unity canonically isomorphic to $1 \otimes I_{|P_2(d)|}$ where $1 \in \Weylalg$ and $I_{|P_2(d)|} \in \Mat_{|P_2(d)|}$ the identity matrix.
\end{proof}

From Proposition \ref{mode-transition}, we have that the mode transition algebra $ \mathfrak{A}_d(\WeylVA)$  is unital. We now prove that the unities in $\mathfrak{A}_d(\WeylVA)$, for each $d \in \mathbb{Z}_{\geq 0}$, form a family of strong unities in the sense of Definition~\ref{def:strong_id}.  

\begin{thm}\label{weyl-unity-thm}
The unity elements $\mathscr{I}_d =1\otimes I_{|P_2(d)|}$ of the mode transitions algebras $\mathfrak{A}_d(\WeylVA) \cong \Weylalg\otimes \Mat_{|P_2(d)|}(\mathbb{C})$ form a family of strong unities.  
\end{thm}

\begin{proof} 
 
First, $I_{|P_2(d)|}$ is identified with an element in $\mathfrak{A}_d(\WeylVA)$, via a choice of basis for $P_2(d)$ denoted  $\{({\bf b}^d(j), {\bf u}^d(j))\}_{j = 1}^{|P_2(d)|}$, so that each $({\bf b}^d(j), {\bf u}^d(j))=({\bf m}_k, {\bf n}_\ell)$ for some $({\bf m}_k, {\bf n}_\ell) \in P_2(d)$ with $m_1 + \cdots + m_k + n_1 + \cdots + n_\ell = d$.
Then with respect to this basis  $I_{|P_2(d)|}$ corresponds to 
\[\sum_{j = 1}^{|P_2(d)|} \frac{1}{c({\bf b}^d(j), {\bf u}^d(j))}\epsilon_{(({\bf b}^d(j), {\bf u}^d(j), ({\bf b}^d(j), {\bf u}^d(j))} . \] 

Moreover, every $\mathfrak{a} \in \mathfrak{A}_{n,-m}(\WeylVA)$ can be written as a linear combination of elements of the form $w \otimes \epsilon_{(({\bf b}^n(i), {\bf u}^n(j)), ({\bf b}^m(k), {\bf u}^m(\ell))}$ for $w \in \Weylalg$.
Using Eqn.\ \eqref{circle-product},  we have that  
\begin{eqnarray*} 
\lefteqn{\mathscr{I}_n \star (w \otimes \epsilon_{(({\bf b}^n(i), {\bf u}^n(j)), ({\bf b}^m(k), {\bf u}^m(\ell))})}\\
&=&  (1 \otimes I_{|P_2(n)| }) \star (w \otimes \epsilon_{(({\bf b}^n(i), {\bf u}^n(j)), ({\bf b}^m(k), {\bf u}^m(\ell))})  \\
&=& 1\cdot w \otimes (I_{|P_2(n)|} \star \epsilon_{(({\bf b}^n(i), {\bf u}^n(j)), ({\bf b}^m(k), {\bf u}^m(\ell))}  ) \\
&=& w \otimes \Big(\sum_{s = 1}^{|P_2(n)|} \frac{1}{c({\bf b}^n(s), {\bf u}^n(s))}\epsilon_{(({\bf b}^n(s), {\bf u}^n(s), ({\bf b}^n(s), {\bf u}^n(s))}\star \epsilon_{(({\bf b}^n(i), {\bf u}^n(j)), ({\bf b}^m(k), {\bf u}^m(\ell))}  \Big) \\
&=& w \otimes \frac{1}{c({\bf b}^n(i), {\bf u}^n(j))} \epsilon_{(({\bf b}^n(i), {\bf u}^n(j)), ({\bf b}^n(i), {\bf u}^n(j))) } \star  \epsilon_{(({\bf b}^n(i), {\bf u}^n(j)), ({\bf b}^m(k), {\bf u}^m(\ell))} \\
&=& w \otimes  \epsilon_{(({\bf b}^n(i), {\bf u}^n(j)), ({\bf b}^m(k), {\bf u}^m(\ell))}  .
\end{eqnarray*}
Extending by linearity to any $\mathfrak{a} \in \mathfrak{A}_{n,-m}(\WeylVA)$, we have that $\mathscr{I}_n \star \mathfrak{a} = \mathfrak{a}$.
The fact that $\mathfrak{a} = \mathfrak{a} \star \mathscr{I}_m$ for all $\mathfrak{a} \in \mathfrak{A}_{n,-m}(\WeylVA)$ is proved analogously. 
\end{proof}

In \cite{L1}, it is shown that the level zero Zhu algebra $\hZhu{0}(\WeylVA)$ is the Weyl algebra $\Weylalg$.   
Moreover, preliminary computations with Addabbo in \cite{BBOHPTY} suggested that $\mathrm{A}_1(\WeylVA) \cong \Weylalg \oplus (\Weylalg \otimes \Mat_2(\mathbb{C}))$. Next we note that  Proposition \ref{mode-transition} and  Theorem \ref{thm:unity} allow us to describe 
all higher level Zhu algebras, in particular, verifying the conjecture in \cite{BBOHPTY} for $\mathrm{A}_1(\WeylVA)$.

\begin{cor}\label{Zhu-alg} 
For $d\geq 1$, the $d$-th level Zhu algebra of the Weyl algebra of central charge 2, $\WeylVA$, satisfies  
\begin{equation*}
    \hZhu{d}(\WeylVA) \cong \mathrm{A}_d(\WeylVA) \cong \bigoplus_{j = 0}^d (\Weylalg \otimes \Mat_{|P_2(j)|}(\mathbb{C})),
\end{equation*}
where $|P_2(d)|$ is the number of 2-component multipartitions of $d$.
\end{cor}

\begin{proof}
It follows from Proposition \ref{mode-transition} that the algebras $\mathfrak{A}_d(\WeylVA)$ are unital. We then apply Theorem \ref{thm:unity}.  
\end{proof}

\section{Modules for \texorpdfstring{$\Z$}{Z}-graded vertex algebras}\label{modules-section} 
In this Section, we introduce various notions of modules for a $\mathbb{Z}$-graded vertex algebra $V$.  
However, we will quickly restrict to the types of $V$-modules we are interested in for the rest of the paper, namely $V$-modules which are $\mathbb{Z}_{\geq 0}$-gradable. Then, we recall the induction functor from modules for a Zhu algebra for $V$ to 
$\mathbb{Z}_{\geq 0}$-gradable $V$-modules from \cite{DLM} in the setting of vertex operator algebras which we extend to the setting of $\mathbb{Z}_{\geq 0}$-graded vertex algebras.  Then, we recall Li's $\mathbf{\Delta}$ operator \cite{Li97} in the setting of vertex operator algebras which we extend here to the setting of vertex algebras, as well as the specific operator that gives rise to spectral flow of $V$-modules.  

Finally we recall the notion of weakly interlocked (cf. \cite{BBOHY}) and prove results about when $V$-modules are weakly interlocked when they are induced from Zhu algebras, and also prove a result characterizing whether or not a weakly interlocked modules subjected to Li's $\mathbf{\Delta}$ operator (including the spectral flow operator) will necessarily be weakly interlocked. 

\subsection{Definitions}

Let $V$ be a $\mathbb{Z}$-graded vertex algebra.

\begin{defn}\label{N-gradable-definition}

    A \emph{$\mathbb{Z}_{\geq 0}$-gradable $V$-module} is a weak $V$-module $W$ that is $\mathbb{Z}_{\geq 0}$-gradable, $W = {\bigoplus}_{k \in \mathbb{Z}_{\geq 0}} W(k)$, with $v_m W(k) \subset W(k + \wt(v) - m -1)$ for homogeneous $v \in V$, $m \in \Z$ and $k \in \mathbb{Z}_{\geq 0}$. Without loss of generality, we can and do assume $W(0) \neq 0$, unless otherwise specified.  We say  that non-zero elements of $W(k)$ have \emph{degree} $k \in \mathbb{Z}_{\geq 0}$.  
\end{defn}

Additionally, if $V$ is a conformal $\mathbb{Z}$-graded vertex algebra, we recall the following definition.
\begin{defn}
 A \emph{$\mathbb{Z}_{\geq 0}$-gradable generalized $V$-module} is a $\mathbb{Z}_{\geq 0}$-gradable weak $V$-module $W$ that admits a decomposition into generalized eigenspaces via the spectrum of $L_0 = \omega_1$ as follows: $W={\bigoplus}_{\lambda \in{\C}}W_\lambda$ where $W_{\lambda}=\{w\in W \mid (L_0 - \lambda \, id_W)^j w= 0 \ \mbox{for some $j \in \mathbb{Z}_{>0}$}\}$, and in addition, $W_{\lambda+n}=0$ for a fixed $\lambda$ and all sufficiently small integers $n$. We say that non-zero elements of $W_\lambda$ have \emph{weight} $\lambda \in \mathbb{C}$.

\end{defn}

\begin{rem}
In the definitions above, we have chosen to restrict to $\mathbb{Z}_{\geq 0}$-gradable (generalized) $V$-modules as opposed to $\mathbb{Z}$-gradable, since these correspond to ``positive energy" modules in the sense of physical models, where the grading $\Delta$ ($L_0$ in the conformal case) defines the energy level.  And, since we want $V$ itself to be a $\mathbb{Z}_{\geq 0}$-gradable (generalized) $V$-module,  we restrict to $\mathbb{Z}_{\geq 0}$-gradable vertex algebras going forward.
\end{rem}

\subsection{\texorpdfstring{$\Z_{\geq0}$}{Z+}-gradable modules and Zhu's correspondence}\label{subsec:zhu_correspondence}

Next, we revisit the functors $\Omega_n$ and $\mathscr{L}_n$, for $n \in \mathbb{Z}_{\geq 0}$, originally defined and studied in \cite{DLM} in the context of a vertex operator algebra $V$ and $\mathbb{Z}_{\geq 0}$-gradable $V$-modules. In this work, we extend their definitions to the setting of $\mathbb{Z}_{\geq 0}$-gradable vertex algebras.

Let $W$ be a $\mathbb{Z}_{\geq 0}$-gradable $V$-module, and set
\begin{equation*}
\Omega_n(W) = \{w \in W \; | \; v_jw = 0\;\text{if } \wt (v) -j-1< -n \; 
\text{for }v\in V \text{ of homogeneous weight}\}.
\end{equation*}
Then $\Omega_n(W)$ is an $\HeZhu{n}(V)$-module, via the action $[a] \mapsto o(a) = a_{\wt(a) -1}$ for $a \in V$ and $[a] = a + O_n(V)$.
This was proved in \cite{DLM} for $V$ a vertex operator algebra but their proof can be extended to the setting of $\mathbb{Z}_{\geq 0}$-gradable vertex algebras straightforwardly, replacing $L_{-1}$ and $L_0$ by the translation and the grading operators respectively.
The functor $\Omega_n:W\mapsto\Omega_n(W)$ is then a covariant functor for the category of $\Z_{\geq0}$-gradable $V$-modules to the one of $\Zhu_n(V)$-modules.

\begin{rem} 
The functor $\Omega_n$ from $\mathbb{Z}_{\geq 0}$-gradable $V$-modules to $\Zhu_n(V)$-modules we have defined here is the same functor defined in \cite{DLM} and \cite{BVY}, but is not the functor called $\Omega_n$ in \cite{Eke11}; rather in \cite{Eke11} the functor denoted $\Omega_n$ is the projection functor onto the $n$-th graded subspace of $W$.  Denoting this projection functor by $\Pi_n$, we have that for $W$ a $\mathbb{Z}_{\geq 0}$-gradable $V$-module,  $\Pi_n(W) = W(n)$.
\end{rem}

Given a weak $V$-module $W$, consider the following linear map
\begin{equation}\label{defining-phi}
\begin{aligned}
\psi_W : \quad U & \longrightarrow \End(W)\\
v^1(m_1)v^2(m_2) \dots v^k(m_k) & \longmapsto (v^1_W)_{m_1}(v^2_W)_{m_2} \dots (v^k_W)_{m_k}  ,
\end{aligned}
\end{equation}
where $(v_W)_m$ is the coefficient of $x^{-m-1}$ in the vertex operator action $Y_W(v,x)$ on the module $W$. 
We will often denote $\psi_W$ by $\psi$ and $(v_W)_m$ by $v_m$  if the module is $V$ itself or if $W$ is clearly implied.

The Zhu algebra $\Zhu_n(V)$ can be regarded as a Lie algebra via the bracket $[u,v] = u *_n v - v *_n u$, and then the map $v( \wt(v) -1) \mapsto v + O_n(V)$ is a well-defined Lie algebra epimorphism from $\mathfrak{L}(V)^f(0)$ onto $\Zhu_n(V)$.
Since $\Zhu_n(V)$ is naturally a Lie algebra homomorphic image of $\mathfrak{L}(V)^f(0)$, a $\Zhu_n(V)$-module $Z$ can be lifted to a $\mathfrak{L}(V)^f(0)$-module, and then to a module for 
$\mathcal{P}_n = \bigoplus_{p < -n} \mathfrak{L}(V)^f(p) \oplus \mathfrak{L}(V)^f(0)$ by letting $\mathfrak{L}(V)^f(p)$ acts trivially for $p\neq 0$.  
Set
\[M_n(Z) = \mbox{Ind}_{\mathcal{P}_n}^{\mathfrak{L}(V)^f}(Z) = U(\mathfrak{L}(V)^f)\otimes_{U(\mathcal{P}_n)}Z.\]
We impose a grading on $M_n(Z)$ by giving $Z$ degree $n$, and by letting $M_n(Z)(k)$, for $k\in\Z$, to be the subspace of $M_n(Z)$ induced from $\mathfrak{L}(V)^f$, i.e., $M_n(Z)(k) = U(\mathfrak{L}(V)^f)_{k-n}Z$. 

For $v \in V$, define $Y_{M_n(Z)}(v,x) \in \mathrm{End} (M_n(Z))((x))$ by
\begin{equation}\label{define-Y_M}
Y_{M_n(Z)}(v,x) = \sum_{m\in\mathbb{Z}} v(m) x^{-m-1}
\end{equation} 
and let $W_{Z}$ be the subspace of $M_n(Z)$ spanned linearly by the coefficients of 
\begin{multline}\label{relations-for-M}
(x_0 + x_2)^{\wt(v) + n} Y_{M_n(Z)}(v, x_0 + x_2) Y_{M_n(Z)}(w, x_2) u \\ 
- (x_2 + x_0)^{\wt( v) + n} Y_{M_n(Z)}(Y(v, x_0)w, x_2) u
\end{multline}
for $v,w \in V$, with $v$ homogeneous, and $u \in Z$.  
It is shown in \cite[Theorem 4.1]{DLM} that if $V$ is a vertex operator algebra and $Z$ is a $\Zhu_n(V)$-module which does not factor through $\Zhu_{n-1}(V)$, then 
\[\overline{M}_n(Z) := M_n(Z)/U (\mathfrak{L}(V)^f)W_Z \]
is an $\mathbb{Z}_{\geq 0}$-gradable $V$-module, $\overline{M}_n(Z) = \bigoplus_{k \in \mathbb{Z}_{\geq 0}} \overline{M}_n(Z) (k)$, with $\overline{M}_n(Z) (0)\neq 0$ and $\overline{M}_n(Z) (n) \cong Z$ as $\Zhu_n(V)$-module. Note that the condition that $Z$ itself does not factor through $\Zhu_{n-1}(V)$ is indeed a necessary and sufficient condition for $\overline{M}_n(Z) (0)\neq 0$ to hold. 
This result still holds if we replace $V$ with a $\mathbb{Z}_{\geq0}$-gradable vertex algebra.

It is also observed in \cite{DLM}, in the setting of vertex operator algebras, that $\overline{M}_n(Z)$ satisfies the following universal property:  for any weak $V$-module $M$ and any $\Zhu_n(V)$-module homomorphism $\Psi: Z \longrightarrow \Omega_n(M)$, there exists a unique weak $V$-module homomorphism $\overline{\Psi}: \overline{M}_n(Z) \longrightarrow M$, such that $\overline{\Psi} \circ \iota = \Psi$ where $\iota$ is the natural injection of $Z$ into $\overline{M}_n(Z)$. This follows from the fact that $\overline{M}_n(Z)$ is generated by $Z$ as a weak $V$-module. 
Again, this universal property of $\overline{M}_n(Z)$ still holds if we allow $V$ to be a $\mathbb{Z}_{\geq 0}$-gradable vertex algebra.

Let $Z^* = \mbox{Hom}(Z, \C)$.  As in the construction in \cite{DLM}, $Z^*$ can be extended to $M_n(Z)$ first by inducing on $M_n(Z)(n)$, and then by letting $Z^*$ annihilate $\bigoplus_{k \neq n} M_n(Z)(k)$. 
Indeed, we have that elements of $M_n(Z)(n) = U(\mathfrak{L}(V)^f)_0Z$ are spanned by elements of the form 
\[o_{p_1}(a^1) \dots o_{p_s}(a^s)Z\]
where $s \in \mathbb{Z}_{\geq 0}$, $p_1 \geq \dots \geq p_s\geq -n$, $p_i \neq 0$, $p_1 + \dots + p_s =0$, $a^i \in V$ and $o_{p_i}(a^i) = a^i(\wt(a^i) - 1 - p_i)$. Then, inducting on $s$ by using \cite[Remark 3.3]{DLM} to reduce from length $s$ vectors to length $s-1$ vectors, we have a well-defined action of $Z^*$ on $M_n(Z)(n)$.  

Set $\mathscr{L}_n(Z) = M_n(Z)/\mathcal{J}$ where
\[\mathcal{J} = \{v \in M_n(Z) \, | \, u'(xv) = 0 \;\mbox{for all}\; u' \in U^{*}, x \in U(\hat{V})\}.\]
It is shown in \cite{DLM} that if $V$ is a vertex operator algebra, and the $\Zhu_n(V)$-module $Z$ does not factor through $\Zhu_{n-1}(V)$, then $\mathscr{L}_n(Z)$ is a well-defined $\mathbb{Z}_{\geq 0}$-gradable $V$-module with $\mathscr{L}_n(Z)(0) \neq 0$; in particular, it is shown that $U(\mathfrak{L}(V)^f)W_Z \subset \mathcal{J}$, for $W_Z$ the subspace of $M_n(Z)$ spanned by the coefficients of \eqref{relations-for-M}, i.e., giving the associativity relations for the weak vertex operators on $M_n(Z)$.
We have the following result proved in \cite{BVY}, which is a necessary modification to what was presented in \cite[Theorem 4.2]{DLM}, and its extended version to the setting of $\mathbb{Z}_{\geq 0}$-graded vertex algebras which follows the same proof (up to changing $L_{-1}$ and $L_0$ appropriately here again).
\begin{prop}
\phantom{x}
\begin{enumerate}[leftmargin=*]
    \item\cite{BVY} For $V$ a vertex operator algebra and $n \in \mathbb{Z}_{\geq 0}$, let $Z$ be a nonzero $\Zhu_n(V)$-module such that if $n>0$, then $Z$ does not factor through $\Zhu_{n-1}(V)$. Then $\mathscr{L}_n(Z)$ is a $\mathbb{Z}_{\geq 0}$-gradable $V$-module with $\mathscr{L}_n(Z)(0) \neq 0$.  If we assume further that there is no nonzero submodule of $Z$ that factors through $\Zhu_{n-1}(V)$, then $\Omega_n/\Omega_{n-1}(\mathscr{L}_n(Z)) \cong Z$.
    
    \item For $V$ a $\mathbb{Z}_{\geq 0}$-graded vertex algebra and $n \in \mathbb{Z}_{\geq 0}$, let $Z$ be a nonzero $\Zhu_n(V)$-module such that if $n>0$, then $Z$ does not factor through $\Zhu_{n-1}(V)$. Then $\mathscr{L}_n(Z)$ is a $\mathbb{Z}_{\geq 0}$-gradable generalized  $V$-module with $\mathscr{L}_n(Z)(0) \neq 0$.  If we assume further that there is no nonzero submodule of $Z$ that factors through $\Zhu_{n-1}(V)$, then $\Omega_n/\Omega_{n-1}(\mathscr{L}_n(Z)) \cong Z$.
\end{enumerate} 
\end{prop}

 This leads us to the following result if we restrict to the irreducible setting, establishing a correspondence between certain irreducible modules of $V$ and $\Zhu_{n}(V)$.
The initial version of this theorem was proved by Zhu \cite{Zhu} when $n=0$ and $V$ is a vertex operator algebra. It was then extended to all higher level Zhu algebras (still for a vertex operator algebra) in \cite{DLM}. Similarly to before, the proof of \cite[Theorem 4.9]{DLM} generalizes to any $\Z_{\geq0}$-graded vertex algebra.

\begin{thm}\label{thm:zhu_correspondence}
Let $V$ be a $\Z_{\geq0}$-graded vertex algebra. Then for $n \in \mathbb{Z}_{\geq 0}$, there is a bijection between isomorphism classes of irreducible $\Zhu_{n}(V)$-modules that cannot factor through $\Zhu_{n-1}(V)$ and irreducible $\mathbb{Z}_{\geq 0}$-gradable $V$-modules.   
\end{thm}

\begin{rem}
A proof of this correspondence (for $n=0$) has also been established in the context of a general vertex algebra in \cite{KDS}.
\end{rem}

\begin{rem} Note that the induction functor $\Phi^L$ from $\Zhu_0(V)$-modules to modules for $\MTA$ of $V$, as well as $\mathfrak{L}^f(V)$-modules, is the same as the  functor $M_0$ in \cite{DLM}. 
\end{rem}

\begin{rem} 
Several papers contain misstatements about higher-level Zhu algebras.
Although \cite[Theorem 1.10]{He} is correct, the preceding statement -- that the induction functor $\mathscr{L}_n$ from the category of $\hZhu{n}(V)$-modules to the category of admissible $V$-modules is an inverse to the restriction functor $\Omega_n/\Omega_{n-1}$ going the other way -- is only valid for \emph{completely reducible} $\mathbb{Z}_{\geq 0}$-gradable $V$-modules. 
It is not true in general that $\Omega_n/\Omega_{n-1}(\mathscr{L}_n(Z)) \cong Z$, even if $Z$ does not factor through $\hZhu{n-1}(V)$.  
The correct statement, given in \cite{BVY}, requires the non-existence of a nontrivial submodule of $Z$ which factors through $\hZhu{n-1}(V)$.
Examples to illustrate the need for this extra condition are given in \cite{BVY,BVY-Virasoro}. 
\end{rem}

\subsection{Li's \texorpdfstring{$\mathbf{\Delta}(x)$}{Δ} operator and spectral flow}\label{subsec:Delta_op}

In this subsection, we recall Li's operator $\mathbf{\Delta}(x) \in \End(V)[[x,x^{-1}]]$ introduced in \cite{Li97} for a vertex operator algebra $V$ and extend it to the setting of vertex algebras. This is the operator that formally gives both simple current extensions of $V$-modules and the notion of spectrally flowed modules.

Note that we are using a bold type $\mathbf{\Delta}$ for Li's operator here since in this paper $\Delta$ denotes the weight operator. 
A lot of properties of this operator proven in \cite{Li97} in the context of vertex operator algebra, naturally extend to the general setting of vertex algebra. For instance, we have the following results, generalizing \cite[Propositions 2.1-2.2]{Li97}.

\begin{prop}\label{Setup for Delta}
Let $V = (V, Y(\cdot, x))$ be a vertex algebra and $T \in \End(V)$ the translation operator. 
\begin{enumerate}[leftmargin=*]
\item  Let $\mathbf{\Delta}(x)\in \End(V)[[x,x^{-1}]]$ be any operator such that 
\begin{equation}\label{Delta1}
\mathbf{\Delta}(x)v\in V[x,x^{-1}],\quad\text{ for all }v\in V.
\end{equation}   
Suppose that $(\tilde{V},\tilde{Y}(\cdot, x)):=(V, Y(\mathbf{\Delta}(x)\cdot, x))$ is a weak $V$-module. Then the following properties hold:
\begin{eqnarray}
\mathbf{\Delta}(x){\bf 1}&=&{\bf 1};\label{Delta2}\\
{[T,\mathbf{\Delta}(x)]}&=&-\frac{d}{dx}\mathbf{\Delta}(x);\label{Delta3}\\
Y(\mathbf{\Delta}(x_2+x_0)v,x_0)\mathbf{\Delta}(x_2)&=&\mathbf{\Delta}(x_2)Y(v,x_0)\text{ for all }v\in V,\label{Delta4}
\end{eqnarray}
where $x_0$ and $x_2$ are formal commuting variables.
Conversely, if the above three properties hold, for any weak $V$-module $(M,Y_M(\cdot,x))$, then $(\tilde{M}, Y_{\tilde{M}}(\cdot,x)):=(M, Y_M(\mathbf{\Delta}(x)\cdot,x))$ is a weak $V$-module.
\bigskip

\item If $V$ is simple and  $\mathbf{\Delta}(x)\in\End(V)[[x,x^{-1}]]$ satisfies properties \eqref{Delta1}-\eqref{Delta4}, then the following holds: 
\[ \mathbf{\Delta}(x)v=0 \mbox{ for $v \in V$ if and only if $v=0$.}\]
\end{enumerate}
\end{prop}

\begin{defn}
    Let $V$ be a vertex algebra and let $\mcU$ be its universal enveloping algebra defined in Section \ref{MTA-Zhu-section} above. We define $G(V)$ to be the set consisting of each $\mathbf{\Delta}(x)=\sum_{r\in \mathbb{C}}\mathbf{\Delta}_rx^r\in \mcU\{x\}$ satisfying the following property: for any weak $V$-module $W$ and any $w\in W$, there exists (finitely many) $n_1,...,n_k\in\mathbb{C}$ such that 
    $$\mathbf{\Delta}(x)w\in x^{n_1}W[x]+...+x^{n_k}W[x],$$
    and the properties  \eqref{Delta1}-\eqref{Delta4} of Proposition \ref{Setup for Delta} hold.
\end{defn}

\begin{prop}\label{hom Delta}
Let $V$ be a vertex algebra and $\mathbf{\Delta}(x) \in G(V)$.
\begin{enumerate}[leftmargin=*]
\item If $M,W$ are (weak) $V$-modules and  $\psi: M \rightarrow W$ is a $V$-module homomorphism, then $\psi: \tilde{M} \rightarrow \tilde{W}$ is also a $V$-module homomorphism.
\item Let $\mathbf{\Delta}(x)\in G(V)$ such that $\mathbf{\Delta}(x)$ has an inverse $\mathbf{\Delta}^{-1}(x)\in \mcU\{x\}$. Then $\mathbf{\Delta}^{-1}(x)\in G(V)$. 
\end{enumerate}
\end{prop}

\begin{proof}
The proof of statement (1), respectively (2), is exactly the same as \cite[Proposition 2.5]{Li97}, respectively \cite[Proposition 2.7]{Li97}, extended in both cases from the setting of vertex operator algebras to that of vertex algebras.  Since property \eqref{Delta4} for $\mathbf{\Delta}^{-1}(x)$ requires a few extra steps involving $\delta$-function identities other than what is given in \cite{Li97}, we detailed this part of the proof here. 

Let $\delta(x) = \sum_{n \in \mathbb{Z}}x^n$.  For formal commuting variables $x_0,x_1, x_2$, we have 
\begin{eqnarray*}
\lefteqn{x_2^{-1} \delta\left(\frac{x_1 - x_0}{x_2}\right) \mathbf{\Delta}^{-1}(x_2) Y(v, x_0) }\\
&=& x_2^{-1} \delta \left( \frac{x_1 - x_0}{x_2}\right) \mathbf{\Delta}^{-1} (x_2) Y(\mathbf{\Delta}(x_1)\mathbf{\Delta}^{-1}(x_1) v, x_0) \mathbf{\Delta}(x_2) \mathbf{\Delta}^{-1} (x_2)\\
&=& x_1^{-1} \delta \left( \frac{x_2 + x_0}{x_1} \right)  \mathbf{\Delta}^{-1} (x_2) Y(\mathbf{\Delta}(x_1)\mathbf{\Delta}^{-1}(x_1) v, x_0) \mathbf{\Delta}(x_2) \mathbf{\Delta}^{-1} (x_2)\\
&=& x_1^{-1} \delta \left( \frac{x_2 + x_0}{x_1} \right)  \mathbf{\Delta}^{-1} (x_2) Y(\mathbf{\Delta}(x_2 + x_0)\mathbf{\Delta}^{-1}(x_2 + x_0) v, x_0) \mathbf{\Delta}(x_2) \mathbf{\Delta}^{-1} (x_2)\\
&=& x_1^{-1} \delta \left( \frac{x_2 + x_0}{x_1} \right)  \mathbf{\Delta}^{-1} (x_2) \mathbf{\Delta}(x_2) Y(\mathbf{\Delta}^{-1}(x_2 + x_0) v, x_0)  \mathbf{\Delta}^{-1} (x_2) \\
&=& x_1^{-1} \delta \left( \frac{x_2 + x_0}{x_1} \right)   Y(\mathbf{\Delta}^{-1}(x_2 + x_0) v, x_0)  \mathbf{\Delta}^{-1} (x_2)\\
&=& x_2^{-1} \delta\left(\frac{x_1 - x_0}{x_2}\right)  Y(\mathbf{\Delta}^{-1}(x_2 + x_0) v, x_0)  \mathbf{\Delta}^{-1} (x_2),
\end{eqnarray*}
giving the result, where we have used %first the fact that $\mathbf{\Delta}(x) \mathbf{\Delta}^{-1}(x) = 1$, then 
the $\delta$-function properties \cite[(2.3.17) and (2.3.21)]{LL} and the property \eqref{Delta4} for $\mathbf{\Delta}(x)$. 
\end{proof}

We let $G^0(V)$ be the subset of all invertible elements of $G(V)$. By Proposition \ref{hom Delta}(2), $G^0(V)$ is a group. Extending \cite{Li97} to the setting of vertex algebras, we have the following.

\begin{thm}\label{first-Delta-thm} 
Let $V$ be a vertex algebra and let $\mathbf{\Delta}(x)\in G^0(V)$. 
Let $W$ be a weak $V$-module and recall that $(\tilde{W},Y_{\tilde{W}}(\cdot,x))=(W,Y_W(\mathbf{\Delta}(x)\cdot,x))$ is also a weak $V$-module. Then
\begin{enumerate}[leftmargin=*]
\item $(W,Y_W)$ is an irreducible $V$-module if and only if $(\tilde{W},Y_{\tilde{W}})$ is. 
\item $(W,Y_W)$ is a maximal $V$-module if and only if $(\tilde{W}, Y_{\tilde{W}})$ is. 
\end{enumerate}
\end{thm}

\begin{proof}  By definition, since $\mathbf{\Delta}(x)\in G^0(V)$,   $\mathbf{\Delta}(x)$ is invertible, with inverse denoted $\mathbf{\Delta}^{-1}(x)$ in $G^0(V)$. Thus on the one hand, if  $(M, Y_M)$ is a submodule of $(W, Y_W)$ then applying $\mathbf{\Delta}(x)$, we have that  $(\tilde{M}, Y_{\tilde{M}})$ is a submodule of $(\tilde{W}, Y_{\tilde{W}})$ so that if $(\tilde{W}, Y_{\tilde{W}})$ is irreducible then $(W, Y_W)$ is.  On the other hand if $(\tilde{M}, Y_{\tilde{M}}(\cdot, x))$ is a submodule of $(\tilde{W}, Y_{\tilde{W}}(\cdot, x)) = (W, Y_W(\mathbf{\Delta}(x) \cdot, x))$ and  $(\tilde{M}, Y_{\tilde{M}}(\cdot, x)) = (M, Y_M(\mathbf{\Delta}(x) \cdot, x))$.  Thus applying $\mathbf{\Delta}^{-1}(x)$, we have that $(M, Y_M(\mathbf{\Delta}^{-1} (x) \mathbf{\Delta}(x)\cdot, x)) = (M, Y_M(\cdot, x))$ is a submodule of $(W, Y_W(\mathbf{\Delta}^{-1}(x) \mathbf{\Delta}(x) \cdot, x)) = (W, Y_W(\cdot, x))$, so that if $(W, Y_W)$ is irreducible then $(\tilde{W}, Y_{\tilde{W}})$ is. This proves part (1).  Part (2) is proved 
similarly. 
\end{proof}

Next we introduce a specific operator $\mathbf{\Delta}(x)\in G^0(x)$ that realizes spectral flow.  This operator was first introduced in \cite{Li97} in the setting of vertex operator algebras.  We show here that many of the results of \cite{Li97} (with minor typos taken into account) still hold for vertex algebras.  

\begin{prop}\label{define-Delta(h,x)} Let $V$ be a vertex algebra  and let $h\in V_1$ satisfying the conditions that $h_nh=\delta_{n,1}\gamma{\bf 1}$, for $n \in \mathbb{Z}_{\geq1}$  where $\gamma \in \mathbb{C}$ is fixed,  and $h_0$ acts semisimply on $V$. We define \begin{equation}\label{define-special-Delta} 
\mathbf{\Delta}(h,x)=x^{h_0}\exp\left(\sum_{k=1}^{\infty}\frac{h_k}{-k}(-x)^{-k}\right).
\end{equation}
Then $\mathbf{\Delta}(h,x)\in G^0(V)$. Moreover, the inverse of $\mathbf{\Delta}(h,x)$ is $\mathbf{\Delta}(-h,x)$.
\end{prop}

\begin{proof}
We first note that by the commutator formula for vertex algebra modes and the property $h_nh = \delta_{n,1} \gamma {\bf 1}$, we have that 
\begin{eqnarray} 
[h_m,h_n] &=& \sum_{j \in \mathbb{Z}_{\geq 0}} \binom{m}{j} (h_jh)_{m+n - j} \ =  (h_0h)_{m+n} + m \gamma {\bf 1}_{m +n - 1} \nonumber \\
&=&  \lambda h_{m+n} + m \gamma \delta_{m+n,0} \label{heisenberg-comm}
\end{eqnarray}
for some $\lambda \in \mathbb{C}$ which is the eigenvalue for $h$ by hypothesis.  In fact $h_0h = 0$, i.e., $\lambda = 0$ in Eqn.\ \eqref{heisenberg-comm} due to the following argument:
Recall that for $u,v\in V$, we have $Y(u,x)v=e^{xT}Y(v,-x)u$. Hence, $u_0v=\sum_{i=0}^{\infty}(-1)^{-i-1}\frac{T^i}{i!}v_iu$. In particular, we have $h_0h=-h_0h+\sum_{i=1}^{\infty}(-1)^{i-1}\frac{T^i}{i!}h_ih$. Since $h_nh=\delta_{n,1}\gamma{\bf 1}$ for $n\geq 1$ and $T({\bf 1})=0$, we then have that $h_0h=-h_0h$. Therefore, $h_0h=0$. 
Thus, in particular, by Eqn.\ \eqref{heisenberg-comm}, $[h_0, h_k] = 0$ for all $k \in \mathbb{Z}_{>0}$. Then
writing $\mathbf{\Delta}(h,x) = x^{h_0} \exp(A)$, we have $\mathbf{\Delta}(-h,x) = x^{-h_0} \exp(-A)$, and we also have that $A, -A$ and $h_0$ all pairwise commute.  Thus  
\[ \mathbf{\Delta}(h,x) \mathbf{\Delta}(-h,x) = x^{h_0} \exp(A) x^{-h_0} \exp(-A) = x^{h_0} x^{-h_0} \exp(A) \exp(-A)  = 1.\]
Thus $\mathbf{\Delta}(h,x)$ is invertible with inverse $\mathbf{\Delta}(-h,x)$. We still need to prove that properties \eqref{Delta1}--\eqref{Delta4} hold for $\mathbf{\Delta}(h,x)$. The fact that $\mathbf{\Delta}(h,x)v \in V[x,x^{-1}]$ for all $V$ follows from the truncation property for $V$.
The fact that 
$\mathbf{\Delta}(h,x){\bf 1} = {\bf 1}$ follows from the creation property, i.e., $h_{k}{\bf 1} = 0$ for $k \geq 0$. To prove $[T, \mathbf{\Delta}(h,x)] = - \frac{d}{dx} \mathbf{\Delta}(h,x)$ we first note that by the $T$-bracket derivative property, we have that 
\[[T,h_k] = -k h_{k-1} \quad \mbox{for } k \in \mathbb{Z}.\]
Next recall that since $\ad(T)$ is a derivation, we have $[T,e^A] = [T,A]e^A$, and thus
\begin{eqnarray*}
[T, \mathbf{\Delta}(h,x)] &=& [T, x^{h_0}] e^A + x^{h_0} [T, e^A ] \\
&=& 0 + x^{h_0} [T,A]e^A\\
&=& x^{h_0} \left( \sum_{k = 1}^\infty h_{k-1} (-x)^{-k} \right) \exp\left(\sum_{k=1}^{\infty}\frac{h_k}{-k}(-x)^{-k}\right)\\
&=& x^{h_0}\left( -h_0 x^{-1} + \sum_{k = 1}^\infty h_k (-x)^{-k-1} \right) e^A \\
&=& - \frac{d}{dx} \mathbf{\Delta}(h,x).
\end{eqnarray*}
Finally, we note that property 
(\ref{Delta4}) is given by \cite[Lemma 3.3]{Li97} the proof of which follows as written in the case of a vertex algebra.  
\end{proof}

Note that if we allow $h_0$ to act semisimply on $V$ with eigenvalues in $\frac{1}{t}\mathbb{Z}$ for $t \in \mathbb{Z}_{>1}$, then as is shown in \cite{Li96-2}, if $(W, Y_W)$ is a weak $V$-module then $(\tilde{W}, Y_{\tilde{W}})$ is a weak $\sigma_h$-twisted $V$-module for $\sigma_h = e^{2\pi i h_0}$. Whereas if we allow only integer values, then $(\tilde{W}, Y_{\tilde{W}})$ is called a \emph{spectrally flowed weak $V$-module}.  The spectrally flowed $V$-modules are also often called ``twisted" modules, following the classical Lie theory terminology, but are not twisted $V$-modules in the vertex algebra theory sense.   

In \cite{Li96-2, Li97} both twisted and spectral flow settings are studied for, for instance, vertex operator algebras associated to affine Lie algebras and lattices.  In Section \ref{Weyl-spectral-flow} we will study certain $\WeylVA$-modules under spectral flow for $\WeylVA$ the Weyl vertex algebra with central charge $c = 2$.  
  
\subsection{Weakly interlocked modules}\label{subsec:weakly_interlocked}
One aspect of the representation theory of $\WeylVA$ we are interested in in the next section is the study of \emph{weakly interlocked modules}.
This property reflects the degree of rigidity of the internal structure of a module $M$ over a ring. It involves two specific submodules of $M$: the socle and the radical.
Recall that the \emph{socle} of a module $M$, denoted by $\Soc(M)$, is the sum of all its irreducible submodules, while the \emph{radical} of $M$, denoted by $\Rad(M)$, is the intersection of all its maximal proper submodules.
Then, we have the following definitions.

\begin{defn}\label{weakly-interlocked-def}
\phantom{x}
\begin{enumerate}[leftmargin=*,label=(\roman*)]
\item Let $A$ be an associative algebra and $\overline{W}$ be an $A$-module.  We say that $\overline{W}$ is \emph{weakly interlocked} if $\overline{W}/\Soc(\overline{W}) \cong \Rad(\overline{W})$ and $\overline{W}/\Rad(\overline{W}) \cong \Soc(\overline{W})$.   
\item Similarly, for $V$ a vertex algebra and $W$ a weak $V$-module, $W$ is said \emph{weakly interlocked} if $W/\Soc(W) \cong \Rad(W)$ and $W/\Rad(W) \cong \Soc(W)$.   
\end{enumerate}
\end{defn}

\begin{rem}
This notion of weakly interlocked is a weaker notion than that of interlocked defined in \cite{Miy1} (see also \cite{BBOHY}) and \emph{strongly interlocked} defined in \cite{BBOHY}. 
These are settings in which \emph{pseudo-traces} have been shown to be well defined for certain vertex operator algebras and their modules.  The notion of weakly interlocked is a minimal condition for any of the known settings in which well-defined pseudo-traces occur.
\end{rem}

We have the following result about $V$-modules induced from Zhu algebras:

\begin{thm}\label{Zhu-weakly-interlocked-thm}
Let $V$ be a vertex algebra, and $Z$ an indecomposable $\Zhu_d(V)$-module such that no nontrivial submodule factors through $\Zhu_{d-1}(V)$.
If $Z$ is not weakly interlocked, then $\mathscr{L}_d(Z)$ is an indecomposable $\mathbb{Z}_{\geq 0}$-gradable $V$-module that is not weakly interlocked.    
\end{thm}

\begin{proof}
Suppose $Z$ is a $\Zhu_d(V)$-module such that no nontrivial submodule is a module for $\Zhu_{d-1}(V)$. Set $W=\mathscr{L}_d(Z)$. 
Then $W(d) = Z$, and by Theorem \ref{thm:zhu_correspondence}, irreducible submodules of $W$ are in bijective correspondence with submodules $Z'\subset Z$.  
Thus $\mathscr{L}_d(\Soc(Z)) = \Soc(W)$.

Suppose $Z'$ is a maximal submodule of $Z$.  Then $Z/Z'$ is irreducible and $\mathscr{L}_d(Z/Z') = \mathscr{L}_d(Z)/\mathscr{L}_d(Z')$ is an irreducible $\mathbb{Z}_{\geq 0}$-gradable $V$-module, implying that $\mathscr{L}_d(Z')$ is a maximal submodule of $W$.  This implies that $\mathscr{L}_d(\Rad(Z)) = \Rad(W)$ and thus $\Rad(W)(d) = \Rad(Z)$.  

Therefore if $W$ is weakly interlocked, then $W(d)/\Soc(W)(d) \cong (W/\Soc(W))(d) \cong (\Rad(W))(d)$ implying that $Z/\Soc(Z) \cong \Rad(Z)$. Similarly $W(d)/\Rad(W)(d) \cong (W/\Rad(W))(d) \cong \Soc(W)(d)$ implying that $Z/\Rad(Z) \cong \Soc(Z)$, and thus $Z$ is weakly interlocked.
\end{proof}

\begin{rem}
    The converse of Theorem \ref{Zhu-weakly-interlocked-thm} is not true, as illustrated in \cite{BBOHY} with the example of $V$ the universal Virasoro vertex operator algebra, and $W$ a $\mathbb{Z}_{\geq 0 }$-gradable $V$-module induced from the level zero Zhu algebra.  If $V$ has central charge $c$ and $Z$ is an indecomposable reducible module for the level zero Zhu algebra with conformal weight $h$,  then $Z$ will be weakly interlocked, but the resulting induced module $W = \mathscr{L}_0(Z)$ will not necessarily be weakly interlocked.  For instance, if $(c,h)$ is in the extended Kac table with $c\neq 1,25$, then $W = \mathscr{L}_0(Z)$ will not be weakly interlocked. 
\end{rem}

We have the following results for modules arising from twisting the action by the operator $\mathbf{\Delta}(x)\in G^0(V)$.

\begin{thm}\label{spectral-flow-thm} Let $V$ be a vertex algebra, $W$  a weak $V$-module, and $\mathbf{\Delta}(x)\in G^0(V)$.  
Then \begin{enumerate}[leftmargin=*]
\item $(M, Y_M)$ is a socle of $(W,Y_W)$ if and only if  $(\tilde{M}, Y_{\tilde{M}})$ is a socle of $(\tilde{W},Y_{\tilde{W}})$.
\item $(M, Y_M)$ is a radical of $(W,Y_W)$ if and only if  $(\tilde{M}, Y_{\tilde{M}})$ is a radical of $(\tilde{W},Y_{\tilde{W}})$.
\item $(W,Y_W)$ is weakly interlocked if and only if $(\tilde{W},Y_{\tilde{W}})$ is weakly interlocked.
\end{enumerate}
In particular, if $\mathbf{\Delta}(x) = \mathbf{\Delta}(h,x)$ as defined in Proposition \ref{define-Delta(h,x)}, then (1)-(3) hold. 
\end{thm}

\begin{proof} 
The statements (1) and (2) follow immediately from (1) and (2) of Theorem \ref{first-Delta-thm}. 
It remains to  prove (3).
For a weak $V$-module $(W,Y_W)$ and a submodule $(M,Y_M)$, recall that the quotient $(W/M,Y_{W/M})$ is again a  weak $V$-module with the action defined by $Y_{W/M}(v,z)(w+M)=Y_W(v,z)w+M$ for $v\in V$, $w\in W$. 
Hence, $(\widetilde{W/M},Y_{\widetilde{W/M}})$ is a weak $V$-module. We have $\widetilde{W/M}=\tilde{W}/\tilde{M}$ and $Y_{\widetilde{W/M}}(v,z)(w+\tilde{M})=Y_{\tilde{W}}(v,z)w+\tilde{M}$ for $v\in V$, $w\in \tilde{W}$.

Now, assume that $(W,Y_W)$ is a weakly interlocked  weak $V$-module; that is  $W$ satisfies $W/\Soc(W)\cong \Rad(W)$ and $W/\Rad(W)\cong\Soc(W)$ as $V$-modules. 
By Proposition \ref{hom Delta}, we have
    \[\begin{gathered}
        \widetilde{\Rad(W)}\cong \widetilde{W/\Soc(W)}=\tilde{W}/\widetilde{\Soc(W)}\\
        \widetilde{\Soc(W)}\cong \widetilde{W/\Rad(W)}=\tilde{W}/\widetilde{\Rad(W)}.
    \end{gathered}\]
Since $\widetilde{\Soc(W)}=\Soc(\tilde{W})$ and $\widetilde{\Rad(W)}=\Rad(\tilde{W})$ by (1) and (2), we have
\[\Rad(\tilde{W})\cong\tilde{W}/\Soc(\tilde{W})\quad\text{and}\quad\Soc(\tilde{W})\cong\tilde{W}/\Rad(\tilde{W}).\]
Therefore, $\tilde{W}$ is weakly interlocked.
Similarly, using the fact that $\mathbf{\Delta}(x)$ is invertible, it follows that if $\tilde{W}$ is weakly interlocked, then $W$ is weakly interlocked. 
\end{proof}

\begin{rem}  In Section \ref{Weyl-spectral-flow}, we apply Theorem \ref{spectral-flow-thm} to modules for the Weyl vertex algebra at central charge 2.  This Theorem will also apply to, for instance, the settings of affine vertex algebras and $\mathcal{W}$-algebras (cf. \cite{CR12, CR, CR13, ACKR,  AKR21, NOHRCW, CMY, C, FRR, ACK} among many others) where spectral flow of modules figures prominently, and graded pseudo-traces would be of interest.    
\end{rem}

\section{Weight Modules for the Weyl vertex algebra at \texorpdfstring{$c = 2$}{c=2}}\label{Weyl-modules-first-section}

In this section, we focus on the classification of $\mathbb{Z}_{\geq0}$-gradable $\WeylVA$-modules  that admit a weight space decomposition with respect to the Heisenberg field $J$ defined in \eqref{eq:heisenberg_field} and are either induced from a Zhu algebra for $\WeylVA$, or are the spectral flow of a module induced from a Zhu algebra. These two categories of modules are related to those studied in, for instance \cite{RW, AW}, as we note below.  

We prove that none of the indecomposable reducible $\WeylVA$-modules induced from the Zhu algebras or realized under spectral flow are weakly interlocked in the sense of Definition \ref{weakly-interlocked-def}. 
This confirms that the category $\mathscr{F}$ studied in \cite{RW, AW} along with its doubly-graded traces is the correct setting for modularity for $\WeylVA$ as shown in those works.

\subsection{Classification of \texorpdfstring{$\WeylVA$}{V}-modules induced from Zhu algebras}
 We begin by focusing on $\mathbb{Z}_{\geq 0}$-gradable weight modules of $\WeylVA$ which are induced from the Zhu algebras. (Note that in general for any $V$, modules induced from a Zhu algebra are $\mathbb{Z}_{\geq 0}$-gradable.)
In fact, Corollary \ref{Zhu-alg} ensures that it suffices to consider only the level zero Zhu algebra of $\WeylVA$. 

\begin{prop}\label{morita-prop}
All $\mathbb{Z}_{\geq 0}$-gradable $\WeylVA$-modules that are induced from a $\hZhu{n}(\WeylVA)$-module ($n\geq0$) are induced from the level zero Zhu algebra, i.e., are induced from modules for the Weyl algebra $\Weylalg$. 
\end{prop}

\begin{proof} Let $W$ be a  $\mathbb{Z}_{\geq 0}$-gradable $\WeylVA$-module induced at level $n$, then $W = \mathscr{L}_n(Z)$ for some $\hZhu{n}(\WeylVA)$-module $Z$.  

Assume first, that no nontrivial proper submodule of $Z$ factors through $\hZhu{n-1}(\WeylVA)$,  then by Corollary \ref{Zhu-alg}, $Z$ is a module for $\Weylalg \otimes \Mat_{|P_2(n)|}(\mathbb{C})$. Then $Z$ is Morita equivalent to a $\Weylalg$-module, and we also have that $\Weylalg \cong \Zhu_0(\WeylVA)$.  In fact, we claim that this Morita equivalence is realized by
\[ Z = W(n) \longleftrightarrow W(0) \] 
where $W(0)$ is the level zero of the module $W$ and 
\[ Z = W(n)  = \WeylVA_n. W(0) \]
where  $\WeylVA_n$ is the $\Mat_{|P_2(n)|}(\mathbb{C})$-module given by the PBW basis vectors of weight $n$ in $\WeylVA$, described in Eqn.\ \eqref{Md}. 
This realization implies that the Zhu-inductions of both $Z$ and $W(0)$ give rise to the same $\WeylVA$-module:
\[W = \mathscr{L}_n(Z) = \mathscr{L}_0(W(0)).\]

If a nontrivial proper submodule of $Z$ factors through $\hZhu{n-1}(\WeylVA)$, then $Z$ decomposes as $Z = Z' \oplus Z''$ where $Z'$ is a $\hZhu{n-1}(\WeylVA)$-modules and $Z''$ is a module for $\Weylalg \otimes \Mat_{|P_2(n)|}(\mathbb{C})$.  
Then $Z'$ corresponds to a $\WeylVA$-module $W'= \mathscr{L}_{n-1}(Z')$ that is induced from the level zero Zhu algebra by the previous discussion, $W'= \mathscr{L}_0(W'(0))$.
Moreover, as in the previous case, $Z''$ is Morita equivalent to a $\Weylalg$-module $W''(0)$ and their inductions give rise to the same $\Weylalg$-module
\[W''= \mathscr{L}_n(Z'') = \mathscr{L}_0(W''(0)).\]
But then
\[W\cong\mathscr{L}_n(Z)= \mathscr{L}_{n-1}(Z') \oplus \mathscr{L}_n(Z'')= W'\oplus W''\]
is induced at level zero since $W'$ and $W''$ are by induction: $W=\mathscr{L}_0(W'(0)\oplus W''(0))$.

\end{proof}

In the following, we will be interested in \emph{weight} $\WeylVA$-modules  which are defined to be (weak) $\WeylVA$-modules on which the zero mode of the Heisenberg field $J=a_{-1}a_0^*\vac$ defined in \eqref{eq:heisenberg_field} acts semi-simply. As noted before, they have been intensively studied in \cite{RW} and \cite{AW}, for instance. 

In regard to Zhu's correspondence, $\Z_{\geq0}$-graded weight $\WeylVA$-modules are induced from \emph{weight $\Weylalg$-modules}, that are  $\Weylalg$-modules on which the product $aa^*\in \Weylalg$ acts semi-simply.
Irreducible weight modules for the Weyl algebra $\Weylalg$ were classified by Block in \cite{Block}. Here we give this classification following the notation of \cite{AW}.

Any irreducible weight $\Weylalg$-module is isomorphic to one and only one of the following:
\begin{enumerate}[leftmargin=*]
    \item $\overline{\WeylVA} = \C[x]$, where $a$ acts as $\frac{\partial}{\partial x}$ and $a^*$ acts as multiplication by $x$.
    \item ${\bf c}\overline{\WeylVA} = \C[x]$, where $a^*$ acts as $-\frac{\partial}{\partial x}$ and $a$ acts as multiplication by $x$. Note that ${\bf c}\overline{\WeylVA}=\varphi_1(\overline{\WeylVA})$, with $\varphi_1$ defined in \eqref{eq:varphi_def} hence, ${\bf c}\overline{\WeylVA}$ is the same as $\overline{\WeylVA}$ as a vector space, but the action of $\Weylalg$ is given by that of $\varphi_1(\Weylalg)$.  Note that $\varphi_1(\overline{\WeylVA}) \cong \varphi_t(\overline{\WeylVA})$ for $t \in \C^\times$ via the vector space isomorphism $p(x) \mapsto tp(x)$ for any $p(x) \in \C[x]$ which gives a $\Weylalg$-module isomorphism. Also for all $t \in \C^\times$, we have that $\varphi_t^2 (\overline{\WeylVA}) \cong \overline{\WeylVA}$ via the $\Weylalg$-module isomorphism $p(x) \mapsto -p(x)$. 
    \item $\overline{\mathcal{W}_{\bar{\lambda}}} = \C[x,x^{-1}]x^\lambda$, for $\bar{\lambda} \in \C/\Z$, $\bar{\lambda}\neq0$, where $a$ acts as $\frac{\partial}{\partial x}$ and $a^*$ acts as multiplication by $x$. Remark that the module ${\bf c}\overline{\mathcal{W}_{\bar{\lambda}}} = \varphi_1 (\overline{\mathcal{W}_{\bar{\lambda}}}) = \C[x,x^{-1}]x^\lambda$ where $a^*$ acts as $-\frac{\partial}{\partial x}$ and $a$ acts as multiplication by $x$, is isomorphic to $\overline{\mathcal{W}_{-\bar{\lambda}}}$ via the $\Weylalg$-module isomorphism $x^{k+\lambda} \mapsto (-1)^k(1+ \lambda) (2 + \lambda) \dots  (k+\lambda) x^{-(k+1 + \lambda)}$.
\end{enumerate}
In addition, $\Weylalg$ admits the following indecomposable reducible modules:
\begin{enumerate}[leftmargin=*,resume]
    \item $\overline{\mathcal{W}^+_0} = \C[x,x^{-1}]$ where $a$ acts as $\frac{\partial}{\partial x}$ and $a^*$ acts as multiplication by $x$.  This module is uniquely characterized by the nonsplit exact sequence
    \[0 \longrightarrow \overline{\WeylVA} \longrightarrow \overline{\mathcal{W}_0^+} \longrightarrow {\bf c} \overline{\WeylVA} \longrightarrow 0 .\]
    \item $\overline{\mathcal{W}^-_0} = \C[x,x^{-1}]$ where $a^*$ acts as $-\frac{\partial}{\partial x}$ and $a$ acts as multiplication by $x$.  This module is uniquely characterized by the nonsplit exact sequence
    \[0 \longrightarrow {\bf c} \overline{\WeylVA} \longrightarrow \overline{\mathcal{W}_0^-} \longrightarrow  \overline{\WeylVA} \longrightarrow 0 .\]
\end{enumerate}
Note that only $\overline{\WeylVA}$ itself is a lowest weight module for $\Weylalg$ with respect to the action of $aa^*$ and corresponding grading, and ${\bf c}\overline{\WeylVA}$ is a highest weight module with respect to this action. The other modules are often called ``relaxed highest weight modules"  for $\Weylalg$, a term taken from \cite{FST}.

 Let $\mathscr{C}_Z$ be the category of $\mathbb{Z}_{\geq 0}$-gradable weight $\WeylVA$-modules that are induced from the Zhu algebras of $\WeylVA$.
By Proposition \ref{morita-prop}, these are all induced at level zero from the isomorphism classes of weight $\Weylalg$-modules.
Hence, inducing the $\Weylalg$-modules listed above in Cases (1)-(3), we have the irreducible weight $\WeylVA$-modules in the category $\mathscr{C}_Z$:
\begin{enumerate}[leftmargin=*]
    \item $\mathcal{V} = \mathscr{L}_0(\overline{\mathcal{V}})$;
    \item $\varphi_1 (\mathcal{V}) = {\bf c} \mathcal{V} = \mathscr{L}_0 ({\bf c} \overline{\mathcal{V}})$;
    \item $\mathcal{W}_{\bar{\lambda}} = \mathscr{L}_0(\overline{\mathcal{W}_{\bar{\lambda}}})$ for $\bar{\lambda} \in \mathbb{C}/\mathbb{Z}$, $\bar{\lambda}\neq0$.
\end{enumerate}
In addition, we have two indecomposable reducible  $\WeylVA$-modules:
\begin{enumerate}[leftmargin=*,resume]
    \item $\mathcal{W}^+_0 = \mathscr{L}_0(\overline{\mathcal{W}^+_0})$ that is uniquely characterized by the nonsplit exact sequence
    \[0 \longrightarrow \mathcal{V} \longrightarrow \mathcal{W}_0^+ \longrightarrow {\bf c} \mathcal{V} \longrightarrow 0 .\]
    \item $\mathcal{W}^-_0 = \mathscr{L}_0(\overline{\mathcal{W}^-_0})$ that is uniquely characterized by the nonsplit exact sequence
    \[0 \longrightarrow {\bf c} \mathcal{V} \longrightarrow \mathcal{W}_0^- \longrightarrow  \mathcal{V} \longrightarrow 0 .\]
\end{enumerate}

\begin{lem}  
All indecomposable reducible weight $\Weylalg$-modules are not weakly interlocked.
\end{lem}

\begin{proof}
Both the indecomposable  reducible  $\Weylalg$-modules $\overline{\mathcal{W}_0^{\pm}}$ are not weakly interlocked precisely because $\overline{\mathcal{V}} \ncong {\bf c} \overline{ \mathcal{V}}$. 
For instance, $\Soc(\overline{\mathcal{W}_0^+}) = \Rad(\overline{\mathcal{W}_0^+}) =  \overline{\mathcal{V}}$ and $\overline{\mathcal{W}_0^+}/\Soc(\overline{\mathcal{W}_0^+}) \cong {\bf c} \overline{\mathcal{V}} \ncong \overline{\mathcal{V}} = \Rad(\overline{\mathcal{W}_0^+})$, proving that $\overline{\mathcal{W}^+_0}$ is not weakly interlocked.  Similarly $\Soc(\overline{\mathcal{W}_0^-}) = \Rad(\overline{\mathcal{W}_0^-}) = {\bf c} \overline{\mathcal{V}}$, and so 
$\overline{\mathcal{W}_0^-}/\Soc(\overline{\mathcal{W}_0^-})\cong  \overline{\mathcal{V}} \ncong {\bf c} \overline{\mathcal{V}} = \Rad(\overline{\mathcal{W}_0^-})$, proving that $\overline{\mathcal{W}^-_0}$ is not weakly interlocked.
\end{proof}

Zhu's correspondence then guarantees that there are no weakly interlocked weight $\WeylVA$-modules induced from weight $\Weylalg$-modules due to Theorem \ref{Zhu-weakly-interlocked-thm}, so that we have the following:

\begin{thm}\label{first-non-interlocked-thm}
All indecomposable reducible weight $\WeylVA$-modules in the category $\mathscr{C}_Z$ are not weakly interlocked.  That is, all indecomposable reducible weight $\WeylVA$-modules induced from a Zhu algebra for $\WeylVA$, are not weakly interlocked.
\end{thm}

\begin{rem}
 In \cite{RW} and \cite{AW}, the category $\mathscr{R}$ of weight $\WeylVA$-modules considered is a subcategory of the $\mathscr{C}_Z$ defined here, as they restrict to real weights to stress the desire to model physical theories.
 Thus Theorem \ref{first-non-interlocked-thm} also implies that none of the indecomposable reducible modules in $\mathscr{R}$ are weakly interlocked, a fact essentially observed in \cite{RW, AW} as well.
\end{rem}

\subsection{Spectral flow of weight \texorpdfstring{ $\mathbb{Z}_{\geq 0}$-gradable $\WeylVA$}{V}-modules induced from Zhu algebras}\label{Weyl-spectral-flow}

In this section we study the indecomposable $\mathbb{Z}_{\geq 0}$-gradable modules for $\WeylVA$ that arise from spectral flow of modules in category $\mathscr{C}_Z$. We denote this category by $\mathscr{C}$ prove that the indecomposable reducible $\WeylVA$-modules in category $\mathscr{C}$ are not weakly interlocked. 
This shows that these modules do not have well-defined doubly-graded pseudo-traces in the sense of extensions of known settings for graded pseudo-traces.

Recall that letting $J:=a_{-1}a^*_0{\bf 1}$, and $J(z)=Y(J,z)=\sum_{n\in\mathbb{Z}}J_nz^{-n-1}$, then $J$ generates a Heisenberg Lie algebra. In particular, $J \in \WeylVA_1$, and $J_nJ = -\delta_{n,1}{\bf 1}$ by Eqn.\ (\ref{eq:heisenberg_field}) and the fact that $J = J_{-1}{\bf 1}$. 
In addition, we have $[J_m,a_n]=-a_{n+m}$ and $[J_m,a^*_n]=a^*_{m+n}$ for $m,n\in\mathbb{Z}$. 
In particular, $J_0$ acts semi-simply with integer eigenvalues.  
Thus by Proposition \ref{define-Delta(h,x)}, for $\ell \in \mathbb{Z}$, we have that $\mathbf{\Delta} (-\ell J,x) \in G^0(\WeylVA)$ defines a spectral flow twist for $\WeylVA$-modules.

Let $(W,Y_W)$ be a $\mathbb{Z}_{\geq 0}$-gradable weight $\mathcal{V}$-module. For fixed $\ell \in \mathbb{Z}$, consider the module 
\begin{equation}
 (\tilde{W},Y_{\tilde{W}}(\cdot,x)) =(W,Y_W(\mathbf{\Delta}(-\ell J,x)\cdot,x)).   
\end{equation}
Analyzing the new action of $\WeylVA$ on $(\tilde{W}, Y_{\tilde{W}})$, we have 
\begin{equation*}
    \mathbf{\Delta}(-\ell J,x)a_{-1}{\bf 1} = x^\ell \exp(\sum_{n=1}^{\infty}\frac{-\ell J_n}{-n}(-x)^{-n})a_{-1}{\bf 1} = x^\ell a_{-1}{\bf 1},
    \end{equation*}
that is
\begin{equation*}
    Y_{\tilde{W}} (a_{-1}{\bf 1},x) = Y_W(\mathbf{\Delta}(- \ell J,x)a_{-1}{\bf 1},x)=Y_W(x^\ell a_{-1}{\bf 1},x),
\end{equation*}  
and
\begin{equation*}
    \mathbf{\Delta}(-\ell J,x)a^*_{0}{\bf 1} = x^{-\ell}\exp(\sum_{n=1}^{\infty}\frac{-\ell J_n}{-n}(-x)^{-n})a^*_{0}{\bf 1}=x^{-\ell} a^*_{0}{\bf 1},
    \end{equation*}
so that 
\begin{equation*}
    Y_{\tilde{W}}(a^*_{0}{\bf 1},x) = Y_W(\mathbf{\Delta}(-\ell J,x)a^*_{0}{\bf 1},x)=Y_W(x^{-\ell}a^*_{0}{\bf 1},x).
\end{equation*}
Thus setting $Y_{\tilde{W}} (v,x)=\sum_{n\in\mathbb{Z}}\tilde{v}_nx^{-n-1}$ and $Y_W(v,x)=\sum_{n\in\mathbb{Z}}v^W_nx^{-n-1}$ for $v \in \WeylVA$, we have \[(\widetilde{a_{-1}{\bf 1}})_n=(a_{-1}{\bf 1})^W_{n+\ell}=a_{n+\ell}^W, \quad \mbox{and} \quad (\widetilde{a^*_{0}{\bf 1}})_n=(a^*_{0}{\bf 1})^W_{n-\ell} = (a^*)^W_{n-\ell}.\]
Therefore, for each $\ell \in \mathbb{Z}$, the module $(\tilde{W}, Y_{\tilde{W}}(\cdot, x)) = (W,Y_W(\mathbf{\Delta}(-\ell J,x)\cdot,x))$ realizes the spectral flow $\sigma^\ell$ given by Eqn. \eqref{spectral-flow}, and we use the notation 
\[(\sigma^\ell(W),Y_\ell(\cdot,x)):=(W,Y_W(\mathbf{\Delta}(-\ell J,x)\cdot,x)).\] 

From Theorem \ref{first-Delta-thm},  we have that any irreducible module in $\mathscr{C}$ is isomorphic to one of the following mutually inequivalent modules
\[ \sigma^\ell (\WeylVA), \qquad \sigma^\ell (\mathcal{W}_{\bar{\lambda}}), \qquad \text{for } \ell \in \mathbb{Z} \text{ and } \bar{\lambda} \in \mathbb{C}/\mathbb{Z} \backslash\{\bar{0}\},\]
and the indecomposable reducible modules in $\mathscr{C}$ are given by
\[ \sigma^\ell(\mathcal{W}_0^{\pm})\quad \mbox{for $\ell \in \mathbb{Z}$}.\]

As a direct consequence of Theorem \ref{spectral-flow-thm}, we have the following.

\begin{cor}\label{non-interlocked-cor}
    All indecomposable reducible weight $\WeylVA$-modules in the category $\mathscr{C}$ are not weakly interlocked.
\end{cor}

\bibliographystyle{alpha}
\bibliography{references}

@article{AKR21,
  title={A realisation of the {B}ershadsky–{P}olyakov algebras and their relaxed modules},
  author={Adamovi\'c, D.  and Kawasetsu, K. and Ridout, D.},
  journal={Lett. Math. Phys},
  fjournal={Letters in Mathematical Physics},
  volume={111},
  number={38},
  year={2021},
  doi={10.1007/s11005-021-01378-1}
}

@article{AB23,
  title={The level two {Z}hu algebra for the {H}eisenberg vertex operator algebra},
  author={Addabbo, D. and Barron, K.},
  journal={Commun. Algebra},
  fjournal={Communications in Algebra},
  volume={51},
  number={8},
  pages={3405--3463},
  year={2023},
  doi={10.1080/00927872.2023.2184638}
}

@article {AB-generaln,
    AUTHOR = {Addabbo, D. and Barron, K.},
     TITLE = {On generators and relations for higher level {Z}hu algebras and applications},
   JOURNAL = {J. Algebra},
  FJOURNAL = {Journal of Algebra},
    VOLUME = {623},
      YEAR = {2023},
     PAGES = {496--540},
       DOI = {10.1016/j.jalgebra.2023.02.023},
}

@article{AW,
  title={Bosonic ghostbusting: The bosonic ghost vertex algebra admits a logarithmic module category with rigid fusion},
  author={Allen, R. and Wood, S.},
  journal={Commun. Math. Phys.},
  fjournal={Communications in Mathematical Physics},
  volume={390},
  pages={959--1015},
  year={2022},
  doi={10.1007/s00220-021-04305-6}
}

@article {ACK,
    AUTHOR = {Arakawa, T. and Creutzig, T. and Kawasetsu, K.},
     TITLE = {Weight representations of affine {K}ac-{M}oody algebras and
              small quantum groups},
   JOURNAL = {Adv. Math.},
  FJOURNAL = {Advances in Mathematics},
    VOLUME = {477},
      YEAR = {2025},
     PAGES = {Paper No. 110365, 48},
      ISSN = {0001-8708,1090-2082},
   MRCLASS = {17B69 (13A50)},
  MRNUMBER = {4912833},
       DOI = {10.1016/j.aim.2025.110365},
       URL = {https://doi-org.proxy.library.nd.edu/10.1016/j.aim.2025.110365},
}

@article {ACKR,
    AUTHOR = {Auger, J. and Creutzig, T. and Kanade, S. and Rupert, M.},
     TITLE = {Braided tensor categories related to {$\mathcal{B}_p$} vertex algebras},
   JOURNAL = {Comm. Math. Phys.},
  FJOURNAL = {Communications in Mathematical Physics},
    VOLUME = {378},
      YEAR = {2020},
     PAGES = {219--260},
       DOI = {10.1007/s00220-020-03747-8}
}

@inproceedings{Barron-varna,
  title={On twisted modules for ${N = 2}$ supersymmetric vertex operator superalgebras},
  author={Barron, K.},
  editor={Dobrev, V.},
  booktitle={Lie Theory and Its Applications in Physics},
  pages={411--420},
  volume={36},
  year={2013},
  organization={Springer Proceedings in Mathematics \& Statistics},
  publisher={Springer, Tokyo},
  doi={10.1007/978-4-431-54270-4_29}
}

@unpublished{BBOHY,
  title={Graded pseudo-traces for strongly interlocked modules for a vertex operator algebra and applications},
  author={Barron, K. and Batistelli, K. and Orosz Hunziker, F. and Yamskulna, G.},
  note={arXiv:2311.17257},
  year={2023}
}

@article{BBOHPTY,
  title={On rationality of $\mathbb{C}$-graded vertex algebras and applications to {W}eyl vertex algebras under conformal flow},
  author={Barron, K. and Batistelli, K. and Orosz Hunziker, F. and Pedi{\'c} Tomi{\'c}, V. and Yamskulna, G.},
  journal={J. Math. Phys.},
  fjournal={Journal of Mathematical Physics},
  volume={63},
  pages={J. Math. Phys. },
  year={2022}
}

@article {BVY,
    AUTHOR = {Barron, K. and Vander Werf, N. and Yang, J.},
     TITLE = {Higher level {Z}hu algebras and modules for vertex operator algebras},
   JOURNAL = {J. Pure Appl. Algebra},
  FJOURNAL = {Journal of Pure and Applied Algebra},
    VOLUME = {223},
      YEAR = {2019},
    NUMBER = {8},
     PAGES = {3295--3317},
       DOI = {10.1016/j.jpaa.2018.11.002}
}

@article{BVY-Virasoro,
  title={The level one {Z}hu algebra for the {V}irasoro vertex operator algebra},
  author={Barron, K. and Vander Werf, N. and Yang, J.},
  journal={Contemp. Math.},
  fjournal={Contemporary Mathematics},
  volume={753},
  pages={17--43},
  year={2020},
  publisher={American Mathematical Society}
}

@inproceedings{Block,
  title={The irreducible representations of the {W}eyl algebra ${A}_1$},
  author={Block, R.},
  booktitle={S{\'e}minaire d'Alg{\`e}bre Paul Dubreil: Lecture Notes in Mathematics},
  pages={69--79},
  year={1979},
  volume={740},
  editor={Malliavin, MP.},
  publisher={Springer, Berlin, Heidelberg}
}

@misc{C,
    AUTHOR = {Creutzig, T.},
     TITLE = {Resolving {V}erlinde’s formula of Logarithmic {CFT}},
      YEAR = {2024},
      note = {arXiv:2411.11383v1},
}

@misc{CMY,
    AUTHOR = {Creutzig, T. and McRae, R. and Yang, J.},
     TITLE = {Ribbon categories of weight modules for affine $\mathfrak{sl}_2$ at admissible levels},
      YEAR = {2024},
      note = {arXiv:2411.11386},
}

@incollection {CR,
    AUTHOR = {Creutzig, T. and Ridout, D.},
     TITLE = {W-algebras extending {$\widehat{\mathfrak{gl}}(1|1)$}},
 BOOKTITLE = {Lie theory and its applications in physics},
    SERIES = {Springer Proc. Math. Stat.},
    VOLUME = {36},
     PAGES = {349--367},
 PUBLISHER = {Springer, Tokyo},
      YEAR = {2013},
      ISBN = {978-4-431-54270-4; 978-4-431-54269-8},
   MRCLASS = {17B67},
  MRNUMBER = {3070663},
       DOI = {10.1007/978-4-431-54270-4\_24},
       URL = {https://doi-org.proxy.library.nd.edu/10.1007/978-4-431-54270-4_24},
}

@article{CR12,
    AUTHOR = {Creutzig, T. and Ridout, D.},
     TITLE = {Modular data and Verlinde formulae for fractional level {WZW} models, {I}},
    journal={Nucl. Phys. B},
  FJOURNAL = {Nuclear Physics B},
  year={2012},
  volume={865},
  pages={83--114},
}

@article{CR13,
    AUTHOR = {Creutzig, T. and Ridout, D.},
     TITLE = {Modular data and Verlinde formulae for fractional level {WZW} models, {II}},
    journal={Nucl. Phys. B},
  FJOURNAL = {Nuclear Physics B},
  year={2013},
  volume={875},
  pages={423--458},
}

@article{DGK,
  title={Conformal blocks on smoothings via mode transition algebras},
  author={Damiolini, C. and Gibney, A. and Krashen, D.},
  journal={Commun. Math. Phys.},
  FJOURNAL = {Communications in Mathematical Physics},
  year={2025},
  volume={406},
  number={131},
  doi={10.1007/s00220-025-05237-1}
}

@article{KDS,
  title={Finite vs. affine {W}-algebras},
  author={De Sole, A. and Kac, V.},
  journal={Jpn. J. Math.},
  fjournal={Japanese Journal of Mathematics},
  volume={1},
  pages={137--261},
  year={2006},
  doi={10.1007/s11537-006-0505-2}
}

@article{D,
    AUTHOR = {Dixmier, J.},
     TITLE = {Sur les alg\`ebres de {W}eyl},
    JOURNAL={Bull. Soc. Math. Fr.},
   FJOURNAL = {Bulletin de la Soci\'{e}t\'{e} Math\'{e}matique de France},
    VOLUME = {96},
      YEAR = {1968},
     PAGES = {209--242}
}

@article {DLM1,
    AUTHOR = {Dong, C. and Li, H. and Mason, G.},
     TITLE = {Simple currents and extensions of vertex operator algebras},
   JOURNAL = {Comm. Math. Phys.},
  FJOURNAL = {Communications in Mathematical Physics},
    VOLUME = {180},
      YEAR = {1996},
    NUMBER = {3},
     PAGES = {671--707},
      ISSN = {0010-3616,1432-0916},
   MRCLASS = {17B69 (81R10)},
  MRNUMBER = {1408523},
MRREVIEWER = {Mirko\ Primc},
       URL = {http://projecteuclid.org.proxy.library.nd.edu/euclid.cmp/1104287460},
}

@article{DLM,
    AUTHOR = {Dong, C. and Li, H. and Mason, G.},
     TITLE = {Vertex operator algebras and associative algebras},
   JOURNAL = {J. Algebra},
  FJOURNAL = {Journal of Algebra},
    VOLUME = {206},
      YEAR = {1998},
    NUMBER = {1},
     PAGES = {67--96},
       DOI = {10.1006/jabr.1998.7425}
}

@inproceedings{DLMM,
    AUTHOR = {Dong, C. and Li, H. and Mason, G. and Montague, P.S.},
    TITLE = {The radical of a vertex operator algebra},
    booktitle = {The Monster and Lie Algebras (Columbus, Ohio, 1996)},
    year = {1998},
    volume={7},
    pages={17--25},
    editors={Ferrar, J. and Harada, K.},
    publisher={Ohio State Univ. Math. Res. Inst. Publ., de Gruyter, Berlin}
}

@article{EG,
  title={Unitarity of rational ${N= 2}$ superconformal theories},
  author={Eholzer, W. and Gaberdiel, M.R.},
  journal={Commun. Math. Phys.},
  fjournal={Communications in mathematical physics},
  volume={186},
  pages={61--85},
  year={1997}
}

@article{Eke11,
    author = {van Ekeren, J.},
    title = {Higher level twisted {Z}hu algebras},
    journal={J. Math. Phys.},
    volume= {52},
    pages={052302},
    year = {2011},
    doi={10.1063/1.3589214}
}

@article{EH19,
    author = {van Ekeren, J. and Heluani, R.},
    title = {A short construction of the {Z}hu algebra},
    journal={J. Algebra},
    fjournal = {Journal of Algebra},
    volume= {528},
    pages= {85--95},
    year = {2019}
}

@article{FRR,
    author = {Fasquel, J. and Raymond, C. and Ridout, D.},
    title = {Modularity of admissible-level $\mathfrak{sl}_3$ minimal models with denominator 2},
    journal = {Commun. Math. Phys.},
    fjournal = {Communications in Mathematical Physics},
    volume={406},
    number={279},
    year = {2025}
}

@article{FST,
    author = {Feigin, B. and Semikhatov, A. and Tipunin, I.},
    title = {Equivalence between chain categories of representations of affine $\mathfrak{sl}(2)$ and ${N=2}$ superconformal algebras},
    journal = {J. Math. Phys.},
    fjournal = {Journal of Mathematical Physics},
    volume={39},
    pages={3865--3905},
    year = {1998}
}

@book{FLanglands,
  title={Langlands correspondence for loop groups},
  author={Frenkel, E.},
  volume={103},
  year={2007},
  publisher={Cambridge University Press, Cambridge}
}

@book{FBbook,
  title={Vertex algebras and algebraic curves},
  author={Frenkel, E. and Ben-Zvi, D.},
  series={Mathematical Surveys and Monographs},
  number={88},
  year={2004},
  publisher={American Mathematical Soc.}
}

@article{FZ,
  title={Vertex operator algebras associated to representations of affine and {V}irasoro algebras},
  author={Frenkel, I.B. and Zhu, Y.},
  journal={Duke Math. J.},
  fjournal={Duke Mathematical Journal},
  volume={66},
  numero={1},
  pages={123--168},
  year={1992}
}

@article {He,
    AUTHOR = {He, X.},
     TITLE = {Higher level {Z}hu algebras are subquotients of universal enveloping algebras},
   JOURNAL = {J. Algebra},
  FJOURNAL = {Journal of Algebra},
    VOLUME = {491},
      YEAR = {2017},
     PAGES = {265--279},
       DOI = {10.1016/j.jalgebra.2017.08.014}
}

@book{LL,
  title={Introduction to vertex operator algebras and their representations},
  author={Lepowsky, J. and Li, H.},
  volume={227},
  year={2004},
  publisher={Springer, Birkhäuser Boston, MA}
}

@article {LMRS,
    AUTHOR = {Lesage, F. and Mathieu, P. and Rasmussen, J. and Saleur, H.},
     TITLE = {The {$\widehat{\mathfrak{su}}(2)_{-1/2}$} {WZW} model and the
              {$\beta\gamma$} system},
   JOURNAL = {Nuclear Phys. B},
  FJOURNAL = {Nuclear Physics. B. Theoretical, Phenomenological, and
              Experimental High Energy Physics. Quantum Field Theory and
              Statistical Systems},
    VOLUME = {647},
      YEAR = {2002},
    NUMBER = {3},
     PAGES = {363--403},
      ISSN = {0550-3213,1873-1562},
   MRCLASS = {81T40 (81R10)},
  MRNUMBER = {1941284},
MRREVIEWER = {Michael\ A. I. Flohr},
       DOI = {10.1016/S0550-3213(02)00905-7},
       URL = {https://doi-org.proxy.library.nd.edu/10.1016/S0550-3213(02)00905-7},
}

@inproceedings{Li96-2,
    AUTHOR = {Li, H.},
     TITLE = {Local systems of twisted vertex operators, vertex operator superalgebras and twisted modules},
 BOOKTITLE = {Moonshine, the {M}onster, and related topics },
    SERIES = {Contemp. Math.},
    VOLUME = {193},
     PAGES = {203--236},
 PUBLISHER = {Amer. Math. Soc., Providence, RI},
    editor ={Chongying Dong and Geoffrey Mason},
      YEAR = {1996},
       DOI = {10.1090/conm/193/02373}
}

@article{Li97,
  title = {The Physics Superselection Principle in Vertex Operator Algebra Theory},
  author = {Li, H.},
  date = {1997},
  journal = {J. Algebra},
  fjournal={Journal of Algebra},
  volume = {196},
  number={2},
  year = {1997},
  pages = {436--457}
}

@article{LY,
    author = {Li, H. and Yamskulna, G.},
    title = {On certain vertex algebras and their modules associated with vertex algebroids},
    journal={J. Algebra},
    fjournal = {Journal of Algebra},
    volume={283},
    number={1},
    pages={367--398},
    year={2005} 
}

@article{L1,
  title={Invariant chiral differential operators and the ${W_3}$ algebra},
  author={Linshaw, A.R.},
  journal={J. Pure Appl. Algebra},
  fjournal={Journal of Pure and Applied Algebra},
  volume={213},
  number={5},
  pages={632--648},
  year={2009}
}

@article {MO,
    AUTHOR = {Maldacena, J. and Ooguri, H.},
     TITLE = {Strings in {$\rm AdS_3$} and the {${\rm SL}(2,{\bf R})$} {WZW}
              model. {I}: {T}he spectrum},
   JOURNAL = {J. Math. Phys.},
  FJOURNAL = {Journal of Mathematical Physics},
    VOLUME = {42},
      YEAR = {2001},
    NUMBER = {7},
     PAGES = {2929--2960},
      ISSN = {0022-2488,1089-7658},
   MRCLASS = {81T30 (81R10 81T40 83E30)},
  MRNUMBER = {1840325},
MRREVIEWER = {Jeremy\ Michelson},
       DOI = {10.1063/1.1377273},
       URL = {https://doi-org.proxy.library.nd.edu/10.1063/1.1377273},
}

@article{Miy1,
  title={({P}seudo-)trace functions and modular invariance of vertex operator algebra},
  author={Miyamoto, M.},
  journal={Pro. Inst. Mat. NAS Ukr.},
  fjournal={Proceedings of Institute of Mathematics of NAS of Ukraine},
  volume={50},
  number={3},
  pages={1145--1151},
  year={2004}
}

@inproceedings{Miy2,
  title={{$C_1$}-cofiniteness and Fusion Products of Vertex Operator Algebras},
  author={Miyamoto, M.},
  editor={Bai, C. and Fuchs, J. and Huang, YZ. and Kong, L. and Runkel, I. and Schweigert, C.},
  booktitle={Conformal Field Theories and Tensor Categories. Mathematical Lectures from Peking University},
  pages={271--279},
  year={2014},
  publisher={Springer, Berlin, Heidelberg}
}

@article{NT,
    author = {Nagatomo, K. and Tsuchiya, A.},
    title = {Conformal field theories associated to regular chiral vertex operator algebras, {I}: Theories over the projective line},
    journal = {Duke Math. J.},
    volume= {128},
    number={3},
    pages={393--471},
    year = {2005}
}

@misc{NOHRCW,
    author = {Nakano, H. and Orosz  Hunkizer, F. and Ros Camacho, A. and Wood, S.},
    title = {Fusion rules and rigidity for weight modules over the simple admissible affine $\mathfrak{sl}(2)$ and $\mathcal{N}=2$ superconformal vertex operator superalgebras},
    year = {2024},
    note={arxiv:2411.11387}
}

@article {R,
    AUTHOR = {Ridout, D.},
     TITLE = {{$\widehat{\mathfrak{sl}}(2)_{-1/2}$}: a case study},
   JOURNAL = {Nuclear Phys. B},
  FJOURNAL = {Nuclear Physics. B. Theoretical, Phenomenological, and
              Experimental High Energy Physics. Quantum Field Theory and
              Statistical Systems},
    VOLUME = {814},
      YEAR = {2009},
    NUMBER = {3},
     PAGES = {485--521},
      ISSN = {0550-3213,1873-1562},
   MRCLASS = {81R10 (17B81 81T40)},
  MRNUMBER = {2513421},
MRREVIEWER = {Markus\ Rosellen},
       DOI = {10.1016/j.nuclphysb.2009.01.008},
       URL = {https://doi-org.proxy.library.nd.edu/10.1016/j.nuclphysb.2009.01.008},
}

@article{RW,
  title={Bosonic ghosts at ${c= 2}$ as a logarithmic {CFT}},
  author={Ridout, D. and Wood, S.},
  journal={Lett. Math. Phys.},
  fjournal={Letters in Mathematical Physics},
  volume={105},
  pages={279--307},
  year={2015},
  doi={10.1007/s11005-014-0740-z}
}

@article{RW1,
  title={Relaxed singular vectors, {J}ack symmetric functions and fractional level $\widehat{\mathfrak{sl}}(2)$ models},
  author={Ridout, D. and Wood, S.},
  journal={Nucl. Phys. B.},
  fjournal={Nuclear Physics B},
  volume={894},
  pages={621--664},
  year={2015}
}

@article {SS,
    AUTHOR = {Schwimmer, A. and Seiberg, N.},
     TITLE = {Comments on the {$N=2,3,4$} superconformal algebras in two dimensions},
   JOURNAL = {Phys. Lett. B},
  FJOURNAL = {Physics Letters B},
    VOLUME = {184},
      YEAR = {1987},
    NUMBER = {2-3},
     PAGES = {191--196},
       DOI = {10.1016/0370-2693(87)90566-1}
}

@article{Zhu,
  title={Modular invariance of characters of vertex operator algebras},
  author={Zhu, Y.},
  journal={J. Am. Math. Soc.},
  fjournal={Journal of the American Mathematical Society},
  volume={9},
  number={1},
  pages={237--302},
  year={1996}
}

\end{document}